\documentclass[reqno]{amsart}
\usepackage[left=1in,right=1in,top=1in,bottom=1in]{geometry}
 \usepackage{microtype}
\usepackage{hyperref}
         
\usepackage{amsmath}             
\usepackage{amsfonts}             
\usepackage{amsthm}               
\usepackage{amsbsy}
\usepackage{amssymb}
\usepackage{amsthm}
\usepackage{amsrefs}

\newtheorem{thm}{Theorem}[section]
\newtheorem{lem}[thm]{Lemma}
\newtheorem{prop}[thm]{Proposition}

\theoremstyle{definition}
\newtheorem{defn}[thm]{Definition}

\newtheorem{rem}[thm]{Remark}

\newcommand{\bC}{{\mathbb{C}}}

\newcommand{\bZ}{{\mathbb{Z}}}

\newcommand{\A}{{\mathcal{A}}}

\newcommand{\I}{{\mathcal{I}}}

\newcommand{\M}{{\mathcal{M}}}

\newcommand{\X}{{\mathcal{X}}}

\newcommand{\qqand}{\qquad\text{and}\qquad}

\newcommand{\inv}{\langle -1 \rangle}

\usepackage{tikz}
\usetikzlibrary{shapes,snakes,calc,arrows}
\usetikzlibrary{decorations.pathreplacing,shapes.geometric}
\usetikzlibrary{calc,positioning}
\tikzset{Box/.style={very thick, rounded corners}}
\tikzset{marked/.style={star, star point height = .75mm, star points =5, fill=black,minimum size=2mm, inner sep=0mm} }
\tikzset{verythickline/.style = {line width=7pt}}
\tikzset{thickline/.style = {line width=5pt}}
\tikzset{medthick/.style = {line width=3pt}}
\tikzset{med/.style = {line width=2pt}}
\tikzset{count/.style = {fill=white,circle,draw,thin, inner sep=2pt}}
\tikzset{rcount/.style = {fill=white,rectangle,draw,thin,inner sep=2pt, rounded corners}}
\tikzset{cpr/.style = {draw,fill=white,rectangle,thin, rounded corners}}

\definecolor{ggreen}{HTML}{00BB33}

\begin{document}

\nocite{*}

\title[A Combinatorial Approach to Voiculescu's Bi-Free Partial Transforms]{A Combinatorial Approach to Voiculescu's \\ Bi-Free Partial Transforms}

\author{Paul Skoufranis}
\address{Department of Mathematics, Texas A\&M University, College Station, Texas, USA, 77843}
\email{pskoufra@math.tamu.edu}

\date{\today}

\begin{abstract}
In this paper, we present a combinatorial approach to the 2-variable bi-free partial $S$- and $T$-transforms recently discovered by Voiculescu.  This approach produces an alternate definition of said transforms using $(\ell, r)$-cumulants.
\end{abstract}

\maketitle

\section{Introduction}

Voiculescu introduced the notion of bi-free pairs of faces in \cite{V2014} as a means to simultaneously study left and right actions of algebras on reduced free product spaces.  Substantial recent work has been performed to better understand bi-freeness and its applications (see \cites{CNS2014-1, CNS2014-2, S2014, V2013-2, MN2013, FW2015, GHM2015}). Specifically, Voiculescu developed a 2-variable bi-free partial $R$-transform in \cite{V2013-2} using analytic techniques thereby generalizing his work from \cite{V1986} to the bi-free setting.  A combinatorial proof of the bi-free partial $R$-transform was given in \cite{S2014} using results from \cite{CNS2014-1}.

Recently, for a pair $(a,b)$ of operators in a non-commutative probability space, Voiculescu in \cite{V2015} constructed a 2-variable bi-free partial $S$-transform, denoted $S_{a,b}(z,w)$, to modify his $S$-transform from \cite{V1987} to the bi-free setting.  Using ideas of Haagerup from \cite{H1997}, \cite{V2015} demonstrates that if $(a_1, b_1)$ and $(a_2, b_2)$ are bi-free, then 
\begin{align}
S_{a_1a_2, b_1b_2}(z,w) = S_{a_1, b_1}(z,w) S_{a_2, b_2}(z,w).  \label{eq:S-property}
\end{align}
Furthermore, Voiculescu constructed a 2-variable bi-free partial $T$-transform $T_{a,b}(z,w)$ in order to study the convolution product where additive convolution is used for the left variables and multiplicative convolution is used for the right variables.  In particular, the defining characteristic of $T_{a,b}(z,w)$ is that if $(a_1, b_1)$ and $(a_2, b_2)$ are bi-free, then 
\begin{align}
T_{a_1+a_2, b_1b_2}(z,w) = T_{a_1, b_1}(z,w) T_{a_2, b_2}(z,w).  \label{eq:T-property}
\end{align}

The goal of this paper is to provide a combinatorial proof of Voiculescu's results from \cite{V2015} and is structured as follows.  Section \ref{sec:preliminaries} will establish all preliminary results, background, and notation necessary for the remainder of the paper.  A reader would benefit greatly from knowledge of the combinatorial approach to the free $S$-transform from \cite{NS1997} and knowledge of the combinatorial approach to bi-freeness from \cite{CNS2014-1} (or the summary in \cite{CNS2014-2}).  Section \ref{sec:T} will provide an equivalent description of $T_{a,b}(z,w)$ using $(\ell, r)$-cumulants and will provide a combinatorial proof of equation (\ref{eq:T-property}).  Section \ref{sec:S} will provide an equivalent description of $S_{a,b}(z,w)$ using $(\ell, r)$-cumulants and will provide a combinatorial proof of equation (\ref{eq:S-property}).

It is worth pointing out in this introduction one slight curiosity that has arisen in the study of bi-free pairs of operators.  If $(a_1, b_1)$ and $(a_2, b_2)$ are bi-free pairs of operators, one may ask, ``Which product do we want to consider: $(a_1a_2, b_1b_2)$ or $(a_1a_2, b_2b_1)$?"  This question arises as it is not clear whether to use the usual multiplication or opposite multiplication on the right pair of algebras.  It is not difficult to see that these two pairs have different distributions by results in \cite{CNS2014-1}.  Note \cite{CNS2014-1}*{Theorem 5.2.1} demonstrates that the $(\ell, r)$-cumualnts of $(a_1a_2, b_2b_1)$ can be computed via a convolution product of the $(\ell, r)$-cumulants of $(a_1, b_1)$ and $(a_2, b_2)$ involving a bi-non-crossing Kreweras complement, just as in the free case.  However, the product of Voiculescu's bi-free partial $S$-transforms of $(a_1, b_1)$ and $(a_2, b_2)$ is the bi-free partial $S$-transform of $(a_1a_2, b_1b_2)$.  As we will see in Section \ref{sec:S}, this is not a matter of difference in notation and, as such, one needs to carefully consider which product to use.

\section{Background and Preliminaries}
\label{sec:preliminaries}

In this section, we recall the necessary background required for this paper.  For more background on scalar-valued bi-free probability, we refer the reader to the summary in \cite{CNS2014-2}*{Section 2}.  This section will also serve the purpose of setting notation for the remainder of the paper, which we endeavour to make consistent with \cite{V2015}.  We will treat all series as formal power series, with commuting variables in the multi-variate cases.

\subsection{Free Transforms}

Let $(\A, \varphi)$ be a non-commutative probability space (that is, a unital algebra $\A$ with a linear functional $\varphi : \A \to \bC$ such that $\varphi(I) = 1$) and let $a \in \A$.  The Cauchy transform of $a$ is
\begin{align*}
G_a(z) = \varphi((zI-a)^{-1}) = \frac{1}{z} \sum_{n\geq 0} \varphi(a^n) z^{-n}
\end{align*}
and the moment series of $a$ is
\begin{align*}
h_a(z) = \varphi((I-az)^{-1}) = \sum_{n\geq 0} \varphi(a^n) z^{n} = \frac{1}{z} G_a\left(\frac{1}{z}\right). 
\end{align*}
Recall one defines $K_a(z)$ to be the inverse of $G_a(z)$ in a neighbourhood of $0$ so that $G_a(K_a(z)) = z$.  Thus $R_a(z) = K_a(z) - \frac{1}{z}$ is the $R$-transform of $a$ and 
\begin{align}
h_a\left( \frac{1}{K_a(z)}\right) = K_a(z) G_a(K_a(z)) = z K_a(z).  \label{eq:K-in-moment}
\end{align}
Furthermore, if $\kappa_n(a)$ denotes the $n^{\mathrm{th}}$ free cumulant of $a$ and the cumulant series of $a$ is
\begin{align*}
c_a(z) = \sum_{n\geq 1} \kappa_n(a) z^n, 
\end{align*}
then one can verify that
\begin{align}
1 + c_a(z) = z K_a(z). \label{eq:K-in-cumulant}
\end{align}

To define the $S$-transform of $a$, we assume $\varphi(a) \neq 0$ and let $\psi_a(z) = h_a(z) - 1$.  Since $\psi_a(0) = 0$ and $\psi'_a(z) = \varphi(a) \neq 0$, $\psi_a(z)$ has a formal power series inverse under composition, denoted $\psi^{\inv}_a(z)$.  We define $\X_a(z) = \psi^{\inv}_a(z)$ so that
\begin{align}
h_a(\X_a(z)) = 1 + \psi_a(\X_a(z)) = 1+z. \label{eq:X-in-moment}
\end{align}
The $S$-transform of $a$ is then defined to be
\begin{align}
S_a(z) = \frac{1+z}{z} \X_a(z). \label{eq:S}
\end{align}

\subsection{Free Multiplicative Functions and Convolution}

Let $NC(n)$ denote the lattice of non-crossing partitions on $\{1,\ldots, n\}$ with its usual reverse refinement order, let $0_n$ to denote the minimal element of $NC(n)$, and let $1_n = \{1, 2, \ldots, n\}$ to denote the maximal element of $NC(n)$.  For $\pi, \sigma \in NC(n)$ with $\pi \leq \sigma$, the interval between $\pi$ and $\sigma$, denoted $[\pi, \sigma]$, is the set
\[
[\pi, \sigma] = \{ \rho \in NC(n) \, \mid \, \pi \leq \rho \leq \sigma\}.
\]
A procedure is described in \cite{S1994} which decomposes each interval of non-crossing partitions into a product of full partitions of the form
\[
[0_1, 1_1]^{k_1} \times [0_2, 1_2]^{k_2} \times [0_3, 1_3]^{k_3} \times \cdots
\]
where $k_j \geq 0$.

The incidence algebra of non-crossing partition, denoted $\I(NC)$, is the algebra of all functions
\[
f : \bigcup_{n\geq 1} NC(n) \times NC(n) \to \bC
\]
such that $f(\pi, \sigma) = 0$ unless $\pi \leq \sigma$, equipped with pointwise addition and a convolution product defined by
\[
(f \ast g)(\pi, \sigma) = \sum_{\rho \in [\pi, \sigma]} f(\pi, \rho) g(\rho, \sigma).
\]

Recall $f \in \I(NC)$ is called multiplicative if whenever $[\pi, \sigma]$ has a canonical decomposition $[0_1, 1_1]^{k_1} \times [0_2, 1_2]^{k_2} \times [0_3, 1_3]^{k_3} \times \cdots$, then
\[
f(\pi, \sigma) = f(0_1, 1_1)^{k_1} f(0_2, 1_2)^{k_2} f(0_3, 1_3)^{k_3} \cdots.
\]
Thus the value of a multiplicative function $f$ on any pair of non-crossing partitions is completely determined by the values of $f$ on full non-crossing partition lattices.  We will denote the set of all multiplicative functions by $\M$ and the set all multiplicative functions $f$ with $f(0_1, 1_1) = 1$ by $\M_1$.

If $f, g \in \M$, one can verify that $f \ast g = g \ast f$.  Furthermore, there is a nicer expression for convolution of multiplicative functions.  Given a non-crossing partition $\pi \in NC(n)$, the Kreweras complement of $\pi$, denoted $K(\pi)$, is the non-crossing partition on $\{1, \ldots, n\}$ with non-crossing diagram obtained by drawing $\pi$ via the standard non-crossing diagram on $\{1,\ldots, n\}$, placing nodes $1', 2', \ldots, n'$ with $k'$ directly to the right of $k$, and drawing the largest non-crossing partition on $1', 2', \ldots, n'$ that does not intersect $\pi$, which is then $K(\pi)$.  The following diagram exhibits that if $\pi = \{\{1,6\}, \{2,3,4\}, \{5\}, \{7\}\}$, then $K(\pi) = \{\{1,4,5\}, \{2\}, \{3\}, \{6,7\}\}$.
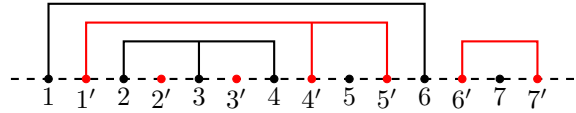
\begin{align*}
	\begin{tikzpicture}[baseline]
	\draw[thick] (1,0) -- (1,1) -- (6,1) -- (6,0);
	\draw[thick] (2,0) -- (2,.5) -- (4,.5) -- (4,0);
	\draw[thick] (3,0) -- (3,.5);
	\draw[thick,red] (1.5,0) -- (1.5,0.75) -- (5.5,0.75)--(5.5, 0);
	\draw[thick,red] (4.5,0) -- (4.5,0.75);
	\draw[thick,red] (6.5,0) -- (6.5,0.5) -- (7.5,0.5)--(7.5, 0);
	\draw[thick, dashed] (0.5,0) -- (8,0);
	\node[below] at (1, 0) {1};
	\draw[fill=black] (1,0) circle (0.05);
	\node[below] at (2, 0) {2};
	\draw[fill=black] (2,0) circle (0.05);
	\node[below] at (3, 0) {3};
	\draw[fill=black] (3,0) circle (0.05);
	\node[below] at (4, 0) {4};
	\draw[fill=black] (4,0) circle (0.05);
	\node[below] at (5, 0) {5};
	\draw[fill=black] (5,0) circle (0.05);
	\node[below] at (6, 0) {6};
	\draw[fill=black] (6,0) circle (0.05);
	\node[below] at (7, 0) {7};
	\draw[fill=black] (7,0) circle (0.05);
	\node[below] at (1.5, 0) {$1'$};
	\draw[red, fill=red] (1.5,0) circle (0.05);
	\node[below] at (2.5, 0) {$2'$};
	\draw[red, fill=red] (2.5,0) circle (0.05);
	\node[below] at (3.5, 0) {$3'$};
	\draw[red, fill=red] (3.5,0) circle (0.05);
	\node[below] at (4.5, 0) {$4'$};
	\draw[red, fill=red] (4.5,0) circle (0.05);
	\node[below] at (5.5, 0) {$5'$};
	\draw[red, fill=red] (5.5,0) circle (0.05);
	\node[below] at (6.5, 0) {$6'$};
	\draw[red, fill=red] (6.5,0) circle (0.05);
	\node[below] at (7.5, 0) {$7'$};
	\draw[red, fill=red] (7.5,0) circle (0.05);
	\end{tikzpicture}
\end{align*}
For $f, g \in \M$, convolution then becomes
\[
(f \ast g)(0_n, 1_n) = \sum_{\pi \in NC(n)} f(0_n, \pi) g(0_n, K(\pi)).
\]
Note \cite{NS1997} demonstrated that if $a, b \in \A$ are free and if $f$ (respectively $g$) is the multiplicative function associated to the cumulants of $a$ (respectively $b$) defined by $f(0_n, 1_n) = \kappa_n(a)$ (respectively $g(0_n, 1_n) = \kappa_n(b)$), then $\kappa_n(ab) = \kappa_n(ba) = (f \ast g)(0_n, 1_n)$.  Furthermore, for $\pi \in NC(n)$ with blocks $\{V_k\}^m_{k=1}$, $f(0_n, \pi) = \kappa_\pi(a) = \prod^m_{k=1} \kappa_{|V_k|}(a)$.

We will need another convolution product on $\M_1$ from \cite{NS1997}.  Let $NC'(n)$ denote all non-crossing partitions $\pi$ on $\{1,\ldots, n\}$ such that $\{1\}$ is a block in $\pi$. It is not difficult to construct an natural isomorphism between $NC'(n)$ and $NC(n-1)$.  The following diagrams illustrate all elements $NC'(4)$, together with their Kreweras complements.
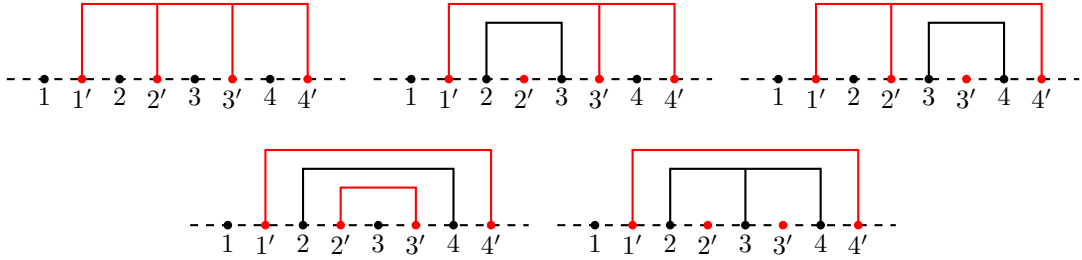
\begin{align*}
	\begin{tikzpicture}[baseline]
	\draw[thick,red] (1.5,0) -- (1.5,1) -- (4.5,1)--(4.5, 0);
	\draw[thick,red] (2.5,0) -- (2.5,1);
	\draw[thick,red] (3.5,0) -- (3.5,1);
	\draw[thick, dashed] (0.5,0) -- (5,0);
	\node[below] at (1, 0) {1};
	\draw[fill=black] (1,0) circle (0.05);
	\node[below] at (2, 0) {2};
	\draw[fill=black] (2,0) circle (0.05);
	\node[below] at (3, 0) {3};
	\draw[fill=black] (3,0) circle (0.05);
	\node[below] at (4, 0) {4};
	\draw[fill=black] (4,0) circle (0.05);
	\node[below] at (1.5, 0) {$1'$};
	\draw[red, fill=red] (1.5,0) circle (0.05);
	\node[below] at (2.5, 0) {$2'$};
	\draw[red, fill=red] (2.5,0) circle (0.05);
	\node[below] at (3.5, 0) {$3'$};
	\draw[red, fill=red] (3.5,0) circle (0.05);
	\node[below] at (4.5, 0) {$4'$};
	\draw[red, fill=red] (4.5,0) circle (0.05);
	\end{tikzpicture}\quad
	\begin{tikzpicture}[baseline]
	\draw[thick,red] (1.5,0) -- (1.5,1) -- (4.5,1)--(4.5, 0);
	\draw[thick,black] (2,0) -- (2,.75) -- (3,.75) -- (3,0);
	\draw[thick,red] (3.5,0) -- (3.5,1);
	\draw[thick, dashed] (0.5,0) -- (5,0);
	\node[below] at (1, 0) {1};
	\draw[fill=black] (1,0) circle (0.05);
	\node[below] at (2, 0) {2};
	\draw[fill=black] (2,0) circle (0.05);
	\node[below] at (3, 0) {3};
	\draw[fill=black] (3,0) circle (0.05);
	\node[below] at (4, 0) {4};
	\draw[fill=black] (4,0) circle (0.05);
	\node[below] at (1.5, 0) {$1'$};
	\draw[red, fill=red] (1.5,0) circle (0.05);
	\node[below] at (2.5, 0) {$2'$};
	\draw[red, fill=red] (2.5,0) circle (0.05);
	\node[below] at (3.5, 0) {$3'$};
	\draw[red, fill=red] (3.5,0) circle (0.05);
	\node[below] at (4.5, 0) {$4'$};
	\draw[red, fill=red] (4.5,0) circle (0.05);
	\end{tikzpicture} \quad
	\begin{tikzpicture}[baseline]
	\draw[thick,red] (1.5,0) -- (1.5,1) -- (4.5,1)--(4.5, 0);
	\draw[thick,red] (2.5,0) -- (2.5,1);
	\draw[thick,black] (3,0) -- (3, .75) -- (4, .75) -- (4,0);
	\draw[thick, dashed] (0.5,0) -- (5,0);
	\node[below] at (1, 0) {1};
	\draw[fill=black] (1,0) circle (0.05);
	\node[below] at (2, 0) {2};
	\draw[fill=black] (2,0) circle (0.05);
	\node[below] at (3, 0) {3};
	\draw[fill=black] (3,0) circle (0.05);
	\node[below] at (4, 0) {4};
	\draw[fill=black] (4,0) circle (0.05);
	\node[below] at (1.5, 0) {$1'$};
	\draw[red, fill=red] (1.5,0) circle (0.05);
	\node[below] at (2.5, 0) {$2'$};
	\draw[red, fill=red] (2.5,0) circle (0.05);
	\node[below] at (3.5, 0) {$3'$};
	\draw[red, fill=red] (3.5,0) circle (0.05);
	\node[below] at (4.5, 0) {$4'$};
	\draw[red, fill=red] (4.5,0) circle (0.05);
	\end{tikzpicture}
\end{align*}
\begin{align*}
	\begin{tikzpicture}[baseline]
	\draw[thick,red] (1.5,0) -- (1.5,1) -- (4.5,1)--(4.5, 0);
	\draw[thick,red] (2.5,0) -- (2.5,.5) -- (3.5, .5) -- (3.5,0);
	\draw[thick,black] (2,0) -- (2,.75) -- (4, .75) -- (4,0);
	\draw[thick, dashed] (0.5,0) -- (5,0);
	\node[below] at (1, 0) {1};
	\draw[fill=black] (1,0) circle (0.05);
	\node[below] at (2, 0) {2};
	\draw[fill=black] (2,0) circle (0.05);
	\node[below] at (3, 0) {3};
	\draw[fill=black] (3,0) circle (0.05);
	\node[below] at (4, 0) {4};
	\draw[fill=black] (4,0) circle (0.05);
	\node[below] at (1.5, 0) {$1'$};
	\draw[red, fill=red] (1.5,0) circle (0.05);
	\node[below] at (2.5, 0) {$2'$};
	\draw[red, fill=red] (2.5,0) circle (0.05);
	\node[below] at (3.5, 0) {$3'$};
	\draw[red, fill=red] (3.5,0) circle (0.05);
	\node[below] at (4.5, 0) {$4'$};
	\draw[red, fill=red] (4.5,0) circle (0.05);
	\end{tikzpicture}
	\quad
	\begin{tikzpicture}[baseline]
	\draw[thick,red] (1.5,0) -- (1.5,1) -- (4.5,1)--(4.5, 0);
	\draw[thick,black] (2,0) -- (2,.75) -- (4, .75) -- (4,0);
	\draw[thick,black] (3,0) -- (3,.75);
	\draw[thick, dashed] (0.5,0) -- (5,0);
	\node[below] at (1, 0) {1};
	\draw[fill=black] (1,0) circle (0.05);
	\node[below] at (2, 0) {2};
	\draw[fill=black] (2,0) circle (0.05);
	\node[below] at (3, 0) {3};
	\draw[fill=black] (3,0) circle (0.05);
	\node[below] at (4, 0) {4};
	\draw[fill=black] (4,0) circle (0.05);
	\node[below] at (1.5, 0) {$1'$};
	\draw[red, fill=red] (1.5,0) circle (0.05);
	\node[below] at (2.5, 0) {$2'$};
	\draw[red, fill=red] (2.5,0) circle (0.05);
	\node[below] at (3.5, 0) {$3'$};
	\draw[red, fill=red] (3.5,0) circle (0.05);
	\node[below] at (4.5, 0) {$4'$};
	\draw[red, fill=red] (4.5,0) circle (0.05);
	\end{tikzpicture}
\end{align*}
We desire to make an observation, which may be proved by induction.  Given two non-crossing partitions $\pi$ and $\sigma$, let $\pi \vee \sigma$ denotes the smallest non-crossing partition larger than both $\pi$ and $\sigma$.  Fix $\pi \in NC'(n)$.  If $\sigma$ is the non-crossing partition on $\{1, 1', 2, 2', \ldots, n, n'\}$ (with the ordering being the order of listing) with blocks $\{k, k'\}$ for all $k$, then the only non-crossing partition $\tau$ on $\{1', \ldots, n'\}$ such that $\pi \cup \tau$ is non-crossing (under the ordering $1, 1', 2, 2', \ldots, n, n'$) and $(\pi \cup \tau) \vee \sigma = 1_{2n}$ is $\tau = K(\pi)$.

For $f, g \in \M_1$, the ``pinched-convolution" of $f$ and $g$, denoted $f \check{\ast} g$, is the unique element of $\M_1$ such that
\[
(f \check{\ast} g)[0_n, 1_n] = \sum_{\pi \in NC'(n)} f(0_n, \pi) g(0_n, K(\pi)).
\]
The pinched-convolution product is not commutative on $\M_1$.

Given an element $f \in \M$, we define the formal power series
\[
\phi_f(z) = \sum_{n\geq 1} f(0_n, 1_n) z^n.
\]
In particular, if $f$ is the multiplicative function associated to the cumulants of $a$ defined by $f(0_n, 1_n) = \kappa_n(a)$, then $\phi_f(z) = c_a(z)$.  Several formulae involving $\phi_f(z)$ are developed in \cite{NS1997}.  In particular, \cite{NS1997}*{Proposition 2.3} demonstrates that if $f,g \in \M_1$ then $\phi_f(\phi_{f \check{\ast} g}(z)) = \phi_{f \ast g}(z)$ and thus
\begin{align}
\phi_{f \check{\ast} g}\left(\phi^{\inv}_{f \ast g}(z)\right) = \phi^{\inv}_f(z).  \label{eq:inversion-with-convolution}
\end{align}
Furthermore, \cite{NS1997}*{Theorem 1.6} demonstrates that
\begin{align}
z \cdot \phi^{\inv}_{f \check{\ast} g}(z) = \phi^{\inv}_{f}(z)\phi^{\inv}_{g}(z).  \label{eq:convolution-with-inverse-series}
\end{align}
A immediate consequence of equation (\ref{eq:convolution-with-inverse-series}) is that if $\varphi(a) = 1$, then
\begin{align}
S_a(z) = \frac{1}{z} c_a^{\inv}(z).  \label{eq:S-to-cumulants}
\end{align}

\subsection{Bi-Freeness}

For a map $\chi : \{1,\ldots, n\} \to \{\ell, r\}$, the set of bi-non-crossing partitions on $\{1,\ldots, n\}$ associated to $\chi$ is denoted by $BNC(\chi)$.  Note $BNC(\chi)$ becomes a lattice where $\pi \leq \sigma$ provided every block of $\pi$ is contained in a single block of $\sigma$. The largest partition in $BNC(\chi)$, which is $\{\{1,\ldots, n\}\}$, will be denoted $1_\chi$.  The work in \cite{CNS2014-1} demonstrates that $BNC(\chi)$ is naturally isomorphic to $NC(n)$ via a permutation of $\{1,\ldots, n\}$ induced by $\chi$.   

Given elements $\{a_n\}^n_{n=1} \subseteq \A$, the $(\ell, r)$-cumulant associated to a map $\chi : \{1,\ldots, n\} \to \{\ell, r\}$ was defined in \cite{MN2013} and will be denoted $\kappa_{\chi}(a_1, \ldots, a_n)$.  Note $\kappa_{\chi}$ is linear in each entry and the main results of \cite{CNS2014-1} is that if $(a_1, b_1)$ and $(a_2, b_2)$ are bi-free two-faced pairs in $(\A, \varphi)$, $\chi : \{1,\ldots, n\} \to \{\ell, r\}$, $\epsilon : \{1, \ldots, n\} \to \{\ell, r\}$, $c_{\ell, k} = a_k$, and $c_{r, k} = b_k$, then 
\[
\kappa_\chi(c_{\chi(1), \epsilon(1)}, \ldots, c_{\chi(n), \epsilon(n)}) = 0
\]
whenever $\epsilon$ is not constant.

Given a $\pi \in BNC(\chi)$, each block $B$ of $\pi$ corresponds to the bi-non-crossing partition $1_{\chi_B}$ for some $\chi_B : B \to \{\ell, r\}$ (where the ordering on $B$ is induced from $\{1,\ldots, n\}$).  We denote
\[
\kappa_\pi(a_1, \ldots, a_n) = \prod_{B \text{ a block of }\pi} \kappa_{1_{\chi_B}}((a_1,\ldots, a_n)|_B)
\]
where $(a_1,\ldots, a_n)|_B$ denotes the $|B|$-tuple where indices not in $B$ are removed.  Similarly, if $V$ is a union of blocks of $\pi$, we denote $\pi|_V$ the bi-non-crossing partition obtained by restricting $\pi$ to $V$.

For $n,m\geq 0$, we will often consider the maps $\chi_{n,m} : \{1,\ldots, n+m\} \to \{\ell, r\}$ such that $\chi(k) = \ell$ if $k \leq n$ and $\chi(k) = r$ if $k > n$.  For notation purposes, it will be useful to think of $\chi_{n,m}$ as a map on $\{1_\ell, 2_\ell, \ldots, n_\ell, 1_r, 2_r, \ldots, m_r\}$ under the identification $k \mapsto k_\ell$ if $k \leq n$ and $k \mapsto (k-n)_r$ if $k > n$. Furthermore, we denote $BNC(n,m)$ for $BNC(\chi_{n,m})$, $1_{n,m}$ for $1_{\chi_{n,m}}$, and, for $n,m\geq 1$, $\kappa_{n,m}(a_1, \ldots, a_n, b_1, \ldots, b_m)$ for $\kappa_{1_{n,m}}(a_1, \ldots, a_n, b_1, \ldots, b_m)$.  Finally, for $n,m \geq 1$, we denote $\kappa_{n,m}(a,b) = \kappa_{1_{n,m}}(a,b)$, $\kappa_{n,0}(a,b) = \kappa_n(a)$, and $\kappa_{0,m}(a,b) = \kappa_{n}(b)$.

\subsection{Bi-Free Transforms}

Given two elements $a,b \in \A$, we define the ordered joint moment and cumulant series of the pair $(a,b)$ to be
\[
H_{a,b}(z,w) = \sum_{n,m\geq 0} \varphi(a^nb^m) z^n w^m \qqand C_{a,b}(z,w) = \sum_{n,m\geq 0} \kappa_{n,m}(a,b) z^n w^m
\]
respectively (where $\kappa_{0,0}(a,b) = 1$).  Note \cite{S2014}*{Theorem 7.2.4} demonstrates that
\begin{align}
h_a(z) + h_b(w) = \frac{h_a(z)h_b(w)}{H_{a,b}(z,w)}  + C_{a,b}(z h_a(z), w h_b(w)) \label{eq:bi-moment}
\end{align}
through combinatorial techniques.  It was also demonstrated that equation (\ref{eq:bi-moment}) was equivalent to Voiculescu's 2-variable bi-free partial $R$-transform from \cite{V2013-2}.

For computational purposes, it will be helpful to consider the series
\begin{align}
K_{a,b}(z,w) = \sum_{n,m\geq 1} \kappa_{n,m}(a,b) z^n w^m = C_{a,b}(z,w) - c_a(z) - c_b(w) - 1. \label{eq:bi-K}
\end{align}
Of use will also be the series
\begin{align}
F_{a,b}(z,w) = \varphi((zI-a)^{-1} (1-wb)^{-1}) = \frac{1}{z} \sum_{n,m \geq 0} \varphi(a^n b^m) z^{-n} w^m = \frac{1}{z} H_{a,b}\left(\frac{1}{z}, w\right).  \label{eq:bi-F}
\end{align}

\subsection{Bi-Free Cumulants of Products}

Of paramount importance to this paper is the ability to write $(\ell, r)$-cumulants of products as sums of $(\ell, r)$-cumulants.  We recall following results from \cite{CNS2014-2}*{Section 9}.  Given two partitions $\pi, \sigma \in BNC(\chi)$, we let $\pi \vee \sigma$ denote the smallest element of $BNC(\chi)$ greater than $\pi$ and $\sigma$.

Let $m,n \geq 1$ with $m < n$ and fix a sequence of integers 
\[
k(0) = 0 < k(1) < \cdots < k(m) = n.
\]
For $\chi : \{1,\ldots, m\} \to \{\ell, r\}$, we define $\widehat{\chi} : \{1,\ldots, n\} \to \{\ell, r\}$ via
\[
\widehat{\chi}(q) = \chi(p_q)
\]
where $p_q$ is the unique element of $\{1,\ldots, m\}$ such that $k(p_q-1) < q \leq k(p_q)$.

There exists an embedding of $BNC(\chi)$ into $BNC(\widehat{\chi})$ via $\pi \mapsto \widehat{\pi}$ where the $p^{\mathrm{th}}$ node of $\pi$ is replaced by the block $\{k(p-1)+1, \ldots, k(p)\}$.
It is easy to see that $\widehat{1_\chi} = 1_{\widehat{\chi}}$ and $\widehat{0_\chi}$ is the partition with blocks $\{\{k(p-1)+1, \ldots, k(p)\} \}_{p=1}^m$.

Using ideas from  \cite{NS2006}*{Theorem 11.12}, \cite{CNS2014-2}*{Theorem 9.1.5} showed that if $\{a_k\}^n_{k=1} \subseteq \A$, then
\begin{align}
\kappa_{1_\chi}\left(a_1 \cdots a_{k(1)}, a_{k(1)+1} \cdots a_{k(2)}, \ldots, a_{k(m-1)+1} \cdots a_{k(m)}\right) = \sum_{\substack{\sigma \in BNC(\widehat{\chi})\\ \sigma \vee \widehat{0_\chi} = 1_{\widehat{\chi}}}} \kappa_\sigma(a_1, \ldots, a_n). \label{eq:product-of-cumulants}
\end{align}

\section{Bi-Free Partial $T$-Transform}
\label{sec:T}

We begin with Voiculescu's bi-free partial $T$-transform as the combinatorics are slightly simpler than the bi-free partial $S$-transform.
\begin{defn}[\cite{V2015}*{Definition 3.1}]
Let $(a, b)$ be a two-faced pair in a non-commutative probability space $(\A, \varphi)$ with $\varphi(b) \neq 0$.  The $2$-variable partial bi-free $T$-transform of $(a, b)$ is the holomorphic function on $(\bC \setminus \{0\})^2$ near $(0,0)$ defined by
\begin{align}
T_{a,b}(z, w) = \frac{w+1}{w} \left( 1-\frac{z}{F_{a,b}(K_a(z), \X_b(w))}\right). \label{eq:T-V}
\end{align}
\end{defn}

It will be useful to note the following equivalent definition of the bi-free partial $T$-transform.  To simplify discussions, we will demonstrate the equality in the case $\varphi(b) = 1$.  This does not hinder the proof of the desired result; that is, Theorem \ref{thm:T-property} (see Remark \ref{rem:T-b=1}).

\begin{prop}
\label{prop:T-new-defn}
If $(a,b)$ is a two-faced pair in a non-commutative probability space $(\A, \varphi)$ with $\varphi(b) = 1$, then, as formal power series,
\begin{align}
T_{a,b}(z,w) = 1 + \frac{1}{w} K_{a,b}\left(z,  c^{\inv}_b(w) \right).  \label{eq:T-my}
\end{align}
\end{prop}
\begin{proof}
Using equations (\ref{eq:K-in-moment}, \ref{eq:X-in-moment}, \ref{eq:bi-moment}), we obtain that
\begin{align*}
\frac{1}{H_{a,b}\left(\frac{1}{K_a(z)}, \X_b(w)\right)} = \frac{1}{zK_a(z)} + \frac{1}{1+w} -\frac{1}{zK_a(z)}\frac{1}{1+w}C_{a,b}\left(z, (1+w)\X_b(w)\right)   .
\end{align*}
Therefore, using equations (\ref{eq:S}, \ref{eq:S-to-cumulants},  \ref{eq:bi-K}, \ref{eq:bi-F}, \ref{eq:T-V}), we obtain that
\begin{align*}
T_{a,b}(z,w) &= \frac{w+1}{w} \left( 1-\frac{z}{\frac{1}{K_a(z)}H_{a,b}\left(\frac{1}{K_a(z)}, \X_b(w)\right)}\right) \\
&= \frac{w+1}{w} \left(1 - zK_a(z) \left( \frac{1}{zK_a(z)} + \frac{1}{1+w} -\frac{1}{zK_a(z)}\frac{1}{1+w}C_{a,b}\left(z, c^{\inv}_b(w)\right)    \right)    \right) \\
&= \frac{1}{w}\left(-zK_a(z) + C_{a,b}\left(z, c^{\inv}_b(w)\right)  \right)\\
&= \frac{1}{w}\left(- zK_a(z) + 1 + c_a(z) + c_b\left( c^{\inv}_b(w) \right) +  K_{a,b}\left(z,  c^{\inv}_b(w) \right)\right)\\
&= \frac{1}{w}\left(w +  K_{a,b}\left(z,  c^{\inv}_b(w) \right)\right) = 1 + \frac{1}{w}  K_{a,b}\left(z,  c^{\inv}_b(w) \right). \qedhere
\end{align*}
\end{proof}

\begin{rem}
\label{rem:T-b=1}
One might be concerned that we have restricted to the case $\varphi(b) = 1$.  However, if we use equation (\ref{eq:T-my}) as the definition of the bi-free partial $T$-transform and if $\lambda \in \bC \setminus \{0\}$, then $T_{a,b}(z,w) = T_{a,\lambda b}(z,w)$.  Indeed $c_{\lambda b}(w) = c_b(\lambda w)$ so $c_{\lambda b}^{\inv}(w) = \frac{1}{\lambda} c^{\inv}_b(w)$.  Therefore, since $\kappa_{n,m}(a, \lambda b) = \lambda^m \kappa_{n,m}(a,b)$, we see that 
\[
K_{a,\lambda b}\left(z,  c^{\inv}_{\lambda b}(w) \right) = K_{a,b}\left(z,  c^{\inv}_b(w) \right).
\]
Thus there is no loss in assuming $\varphi(b) = 1$.
\end{rem}
\begin{rem}
Note Proposition \ref{prop:T-new-defn} immediately provides the $T$-transform part of \cite{V2015}*{Proposition 4.2}.  Indeed if $a$ and $b$ are elements of a non-commutative probability space $(\A, \varphi)$ with $\varphi(b) \neq 0$ and $\varphi(a^nb^m) = \varphi(a^n)\varphi(b^m)$ for all $n,m \geq 0$, then $\kappa_{n,m}(a,b) = 0$ for all $n,m\geq 1$ (see \cite{S2014}*{Section 3.2}).  Hence $K_{a,b}(z,w) = 0$ so $T_{a,b}(z,w) = 1$.
\end{rem}

We desire to prove the following, which was one of two main results of \cite{V2015}, using combinatorics via Proposition \ref{prop:T-new-defn}.  
\begin{thm}[\cite{V2015}*{Theorem 3.1}]
\label{thm:T-property}
Let $(a_1, b_1)$ and $(a_2, b_2)$ be bi-free two-faced pairs in a non-commutative probability space $(\A, \varphi)$ with $\varphi(b_1) \neq 0$ and $\varphi(b_2) \neq 0$.  Then
\[
T_{a_1+ a_2, b_1b_2}(z,w) = T_{a_1, b_1}(z,w) T_{a_2, b_2}(z,w)
\]
on $(\bC \setminus \{0\})^2$ near $(0,0)$.
\end{thm}

To simplify the proof of the result, we will assume that $\varphi(b_1) = \varphi(b_2) = 1$.  Note $\varphi(b_1b_2) = 1$ by freeness of the right algebras in bi-free pairs.  Furthermore, we will let $g_j$ denote the multiplicative function associated to the cumulants of $b_j$ defined by $g_j(0_n, 1_n) = \kappa_n(b_j)$.  Recall if $g$ is the multiplicative function associated to the cumulants of $b_1b_2$, then $g = g_1 \ast g_2$.  Thus $\phi^{\inv}_{g}(w) = c^{\inv}_{b_1b_2}(w)$ and $\phi^{\inv}_{g_j}(w) = c^{\inv}_{b_j}(w)$.

By Proposition \ref{prop:T-new-defn}, it suffices to show that
\begin{align}
K_{a_1+a_2, b_1b_2}\left(z, \phi^{\inv}_{g}(w)\right) =  \Theta_1(z,w) + \Theta_2(z,w) + \frac{1}{w} \Theta_1(z,w)\Theta_2(z,w) \label{eq:T-eq-to-verify}
\end{align}
where
\[
\Theta_j(z,w) = K_{a_j, b_j}\left(z, \phi^{\inv}_{g_j}(w)\right).
\]

Recall
\[
K_{a_1+a_2, b_1b_2}(z,w) = \sum_{n,m\geq 1} \kappa_{n,m}(a_1+a_2, b_1b_2) z^n w^m.
\]
For fix $n,m \geq 1$, let $\sigma_{n,m}$ denote the element of $BNC(n,2m)$ with blocks $\{\{k_\ell\}\}_{k=1}^n \cup \{\{(2k-1)_r, (2k)_r\}\}_{k=1}^m$.  Thus equation (\ref{eq:product-of-cumulants}) implies that
\[
\kappa_{n,m}(a_1+a_2, b_1b_2) = \sum_{\substack{ \pi \in BNC(n, 2m) \\ \pi \vee \sigma_{n,m} = 1_{n,2m} }} \kappa_\pi(\underbrace{a_1+a_2, \ldots, a_1+a_2}_n, \underbrace{b_1, b_2, b_1, b_2, \ldots, b_1, b_2}_{b_1 \text{ occurs }m \text{ times}}).
\]
Notice that if $\pi  \in BNC(n, 2m)$ and $\pi \vee \sigma_{n,m} = 1_{n,2m}$, then any block of $\pi$ containing a $k_\ell$ must contain a $j_r$ for some $j$.  Furthermore, if $1 \leq k < j \leq n$ are such that $k_\ell$ and $j_\ell$ are in the same block of $\pi$, then $q_\ell$ must be in the same block as $k_\ell$ for all $k \leq q \leq j$.  Moreover, since $(a_1, b_1)$ and $(a_2, b_2)$ are bi-free, we note that 
\[
\kappa_\pi(\underbrace{a_1+a_2, \ldots, a_1+a_2}_n, \underbrace{b_1, b_2, b_1, b_2, \ldots, b_1, b_2}_{b_1 \text{ occurs }m \text{ times}}) = 0
\]
if $\pi$ contains a block containing a $(2k)_r$ and a $(2j-1)_r$ for some $k, j$.

For $n,m\geq 1$, let $BNC_T(n,m)$ denote all $\pi \in BNC(n, 2m)$ such that $\pi \vee \sigma_{n,m} = 1_{n,2m}$ and $\pi$ contains no blocks containing both a $(2k)_r$ and a $(2j-1)_r$ for some $k, j$.  Consequently, we obtain
\[
K_{a_1+a_2, b_1b_2}(z,w) = \sum_{n,m\geq 1} \left(\sum_{\pi \in BNC_T(n,m)} \kappa_\pi(\underbrace{a_1+a_2, \ldots, a_1+a_2}_n, \underbrace{b_1, b_2, b_1, b_2, \ldots, b_1, b_2}_{b_1 \text{ occurs }m \text{ times}})   \right) z^n w^m.
\]
We desire to divide up this sum into two parts based on types of partitions in $BNC_T(n,m)$.  Let $BNC_T(n,m)_e$ denote all $\pi \in BNC_T(n,m)$ such that the block containing $1_\ell$ also contains a $(2k)_r$ for some $k$, and let $BNC_T(n,m)_o$ denote all $\pi \in BNC_T(n,m)$ such that the block containing $1_\ell$ also contains a $(2k-1)_r$ for some $k$.  Note $BNC_T(n,m)_e$ and $BNC_T(n,m)_o$ are disjoint and $BNC_T(n,m)_e \cup BNC_T(n,m)_o = BNC_T(n,m)$ by previous discussions.  Therefore, if for $d \in \{o,e\}$ we define
\[
\Psi_d(z,w) := \sum_{n,m\geq 1} \left(\sum_{\pi \in BNC_T(n,m)_d} \kappa_\pi(\underbrace{a_1+a_2, \ldots, a_1+a_2}_n, \underbrace{b_1, b_2, b_1, b_2, \ldots, b_1, b_2}_{b_1 \text{ occurs }m \text{ times}})   \right) z^n w^m,
\]
then
\[
K_{a_1+a_2, b_1b_2}(z,w) = \Psi_e(z,w) + \Psi_o(z,w).
\]
We will derive expressions for $\Psi_e(z,w)$ and $\Psi_o(z,w)$ beginning with $\Psi_e(z,w)$.

\begin{lem}
\label{lem:T-case-1}
Under the above notation and assumptions,
\[
\Psi_e(z,w) = K_{a_2, b_2}\left(z, \phi_{g_2 \check{\ast} g_1}(w) \right). 
\]
\end{lem}
\begin{proof}
For each $n,m\geq 1$, we desire to rearrange the sum in $\Psi_e(z,w)$ by expanding $\kappa_\pi$ as a product of full $(\ell, r)$-cumulants and summing over all $\pi$ with the same block containing $1_\ell$.  

Fix $n,m \geq 1$.  If $\pi \in BNC_T(n,m)_e$, then the block $V_\pi$ containing $1_\ell$ must also contain $(2k)_r$ for some $k$ and thus all of $(2m)_r, 1_\ell, 2_\ell, \ldots, n_\ell$ must be in $V_\pi$ in order for $\pi \vee \sigma_{n,m} = 1_{n,2m}$.  Below is an example of such a $\pi$.  Two nodes can be connected to each other with a solid line if and only if lie in the same block of $\pi$ and two nodes are connected with a dotted line if and only if they are in the same block of $\sigma_{n,m}$.  The condition $\pi \vee \sigma_{n,m} = 1_{n,2m}$ means one may travel from any one node to another using a combination of solid and dotted lines.
Note we really should draw all of the left nodes above all of the right notes, but we do not do so in order to save space.
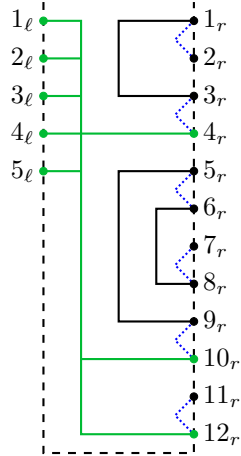
\begin{align*}
	\begin{tikzpicture}[baseline]
	\draw[thick, dashed] (-1,5.75) -- (-1,-.25) -- (1,-.25) -- (1,5.75);
	\draw[thick, blue, densely dotted] (1, 5.5) -- (0.75, 5.25) -- (1, 5);
	\draw[thick, blue, densely dotted] (1, 4.5) -- (0.75, 4.25) -- (1, 4);
	\draw[thick, blue, densely dotted] (1, 3.5) -- (0.75, 3.25) -- (1, 3);
	\draw[thick, blue, densely dotted] (1, 2.5) -- (0.75, 2.25) -- (1, 2);
	\draw[thick, blue, densely dotted] (1, 1.5) -- (0.75, 1.25) -- (1, 1);
	\draw[thick, blue, densely dotted] (1, 0.5) -- (0.75, 0.25) -- (1, 0);
	\draw[ggreen, thick] (-1,5.5) -- (-0.5,5.5) -- (-.5,0) -- (1, 0);
	\draw[thick] (1,1.5) -- (0,1.5) -- (0,3.5) -- (1, 3.5);
	\draw[thick] (1,4.5) -- (0,4.5) -- (0,5.5) -- (1, 5.5);
	\draw[thick] (1,2) -- (0.5,2) -- (0.5,3) -- (1, 3);
	\draw[ggreen, thick] (-1, 5) -- (-0.5, 5);
	\draw[ggreen, thick] (-1, 4.5) -- (-0.5, 4.5);
	\draw[ggreen, thick] (-1, 4) -- (-0.5, 4);
	\draw[ggreen, thick] (-1, 3.5) -- (-0.5, 3.5);
	\draw[ggreen, thick] (1, 4) -- (-0.5, 4);
	\draw[ggreen, thick] (1, 1) -- (-0.5, 1);
	\node[left] at (-1, 5.5) {$1_\ell$};
	\draw[ggreen, fill=ggreen] (-1,5.5) circle (0.05);
	\node[left] at (-1, 5) {$2_\ell$};
	\draw[ggreen, fill=ggreen] (-1,5) circle (0.05);
	\node[left] at (-1, 4.5) {$3_\ell$};
	\draw[ggreen, fill=ggreen] (-1,4.5) circle (0.05);
	\node[left] at (-1, 4) {$4_\ell$};
	\draw[ggreen, fill=ggreen] (-1,4) circle (0.05);
	\node[left] at (-1, 3.5) {$5_\ell$};
	\draw[ggreen, fill=ggreen] (-1,3.5) circle (0.05);
	\draw[fill=black] (1,5.5) circle (0.05);
	\node[right] at (1,5.5) {$1_r$};
	\draw[fill=black] (1,5) circle (0.05);
	\node[right] at (1,5) {$2_r$};
	\draw[fill=black] (1,4.5) circle (0.05);
	\node[right] at (1,4.5) {$3_r$};
	\draw[ggreen, fill=ggreen] (1,0) circle (0.05);
	\node[right] at (1,4) {$4_r$};
	\draw[ggreen, fill=ggreen] (1,4) circle (0.05);
	\node[right] at (1,3.5) {$5_r$};
	\draw[fill=black] (1,3.5) circle (0.05);
	\node[right] at (1,3) {$6_r$};
	\draw[fill=black] (1,3) circle (0.05);
	\node[right] at (1,2.5) {$7_r$};
	\draw[fill=black] (1,2.5) circle (0.05);
	\node[right] at (1,2) {$8_r$};
	\draw[fill=black] (1,2) circle (0.05);
	\node[right] at (1,1.5) {$9_r$};
	\draw[fill=black] (1,1.5) circle (0.05);
	\node[right] at (1,1) {$10_r$};
	\draw[ggreen, fill=ggreen] (1,1) circle (0.05);
	\node[right] at (1,0.5) {$11_r$};
	\draw[fill=black] (1,0.5) circle (0.05);
	\node[right] at (1,0) {$12_r$};
	\end{tikzpicture}
\end{align*}

Let $E = \{(2k)_r\}^{m}_{k=1}$, let $O = \{(2k-1)_r\}^{m}_{k=1}$, let $s$ denote the number of elements of $E$ contained in $V_\pi$ (so $s \geq 1$), and let $1 \leq k_1 < k_2 < \cdots < k_s = m$ be such that $(2k_q)_r \in V_\pi$.  Note $V_\pi$ divides the right nodes into $s$ disjoint regions.  For each $1 \leq q \leq s$, let $j_q = k_q - k_{q-1}$, where $k_0 = 0$, and let $\pi_q$ denote the non-crossing partition obtained by restricting $\pi$ to $\{(2k_{q-1} + 1)_r, (2k_{q-1}+2)_r, \ldots, (2k_{q}-1)_r\}$.  Note that $\sum^s_{q=1} j_q = m$.  Furthermore, if $\pi'_q$ is obtained from $\pi_q$ by adding the singleton block $\{(2k_q)_r\}$, then $\pi'_q|_{E}$ is naturally an element of $NC'(j_q)$ and $\pi'_q|_O$ is naturally an element of $NC(j_q)$, which must be $K(\pi'_q|_E)$ in order for $\pi \vee \sigma_{n,m} = 1_{n,2m}$.  The below diagram demonstrates an example of this restriction.
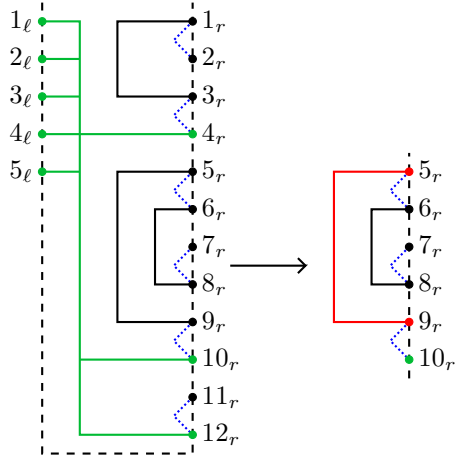
\begin{align*}
	\begin{tikzpicture}[baseline]
	\draw[thick] (1.5, 2.25) -- (2.5, 2.25) -- (2.4, 2.15);
	\draw[thick] (2.5,2.25) -- (2.4, 2.35);
	\draw[thick, dashed] (-1,5.75) -- (-1,-.25) -- (1,-.25) -- (1,5.75);
	\draw[thick, blue, densely dotted] (1, 5.5) -- (0.75, 5.25) -- (1, 5);
	\draw[thick, blue, densely dotted] (1, 4.5) -- (0.75, 4.25) -- (1, 4);
	\draw[thick, blue, densely dotted] (1, 3.5) -- (0.75, 3.25) -- (1, 3);
	\draw[thick, blue, densely dotted] (1, 2.5) -- (0.75, 2.25) -- (1, 2);
	\draw[thick, blue, densely dotted] (1, 1.5) -- (0.75, 1.25) -- (1, 1);
	\draw[thick, blue, densely dotted] (1, 0.5) -- (0.75, 0.25) -- (1, 0);
	\draw[ggreen, thick] (-1,5.5) -- (-0.5,5.5) -- (-.5,0) -- (1, 0);
	\draw[thick] (1,1.5) -- (0,1.5) -- (0,3.5) -- (1, 3.5);
	\draw[thick] (1,4.5) -- (0,4.5) -- (0,5.5) -- (1, 5.5);
	\draw[thick] (1,2) -- (0.5,2) -- (0.5,3) -- (1, 3);
	\draw[ggreen, thick] (-1, 5) -- (-0.5, 5);
	\draw[ggreen, thick] (-1, 4.5) -- (-0.5, 4.5);
	\draw[ggreen, thick] (-1, 4) -- (-0.5, 4);
	\draw[ggreen, thick] (-1, 3.5) -- (-0.5, 3.5);
	\draw[ggreen, thick] (1, 4) -- (-0.5, 4);
	\draw[ggreen, thick] (1, 1) -- (-0.5, 1);
	\node[left] at (-1, 5.5) {$1_\ell$};
	\draw[ggreen, fill=ggreen] (-1,5.5) circle (0.05);
	\node[left] at (-1, 5) {$2_\ell$};
	\draw[ggreen, fill=ggreen] (-1,5) circle (0.05);
	\node[left] at (-1, 4.5) {$3_\ell$};
	\draw[ggreen, fill=ggreen] (-1,4.5) circle (0.05);
	\node[left] at (-1, 4) {$4_\ell$};
	\draw[ggreen, fill=ggreen] (-1,4) circle (0.05);
	\node[left] at (-1, 3.5) {$5_\ell$};
	\draw[ggreen, fill=ggreen] (-1,3.5) circle (0.05);
	\draw[fill=black] (1,5.5) circle (0.05);
	\node[right] at (1,5.5) {$1_r$};
	\draw[fill=black] (1,5) circle (0.05);
	\node[right] at (1,5) {$2_r$};
	\draw[fill=black] (1,4.5) circle (0.05);
	\node[right] at (1,4.5) {$3_r$};
	\draw[ggreen, fill=ggreen] (1,0) circle (0.05);
	\node[right] at (1,4) {$4_r$};
	\draw[ggreen, fill=ggreen] (1,4) circle (0.05);
	\node[right] at (1,3.5) {$5_r$};
	\draw[fill=black] (1,3.5) circle (0.05);
	\node[right] at (1,3) {$6_r$};
	\draw[fill=black] (1,3) circle (0.05);
	\node[right] at (1,2.5) {$7_r$};
	\draw[fill=black] (1,2.5) circle (0.05);
	\node[right] at (1,2) {$8_r$};
	\draw[fill=black] (1,2) circle (0.05);
	\node[right] at (1,1.5) {$9_r$};
	\draw[fill=black] (1,1.5) circle (0.05);
	\node[right] at (1,1) {$10_r$};
	\draw[ggreen, fill=ggreen] (1,1) circle (0.05);
	\node[right] at (1,0.5) {$11_r$};
	\draw[fill=black] (1,0.5) circle (0.05);
	\node[right] at (1,0) {$12_r$};
	\end{tikzpicture}
	\quad
	\begin{tikzpicture}[baseline]
	\draw[thick, dashed]  (1,.75) -- (1,3.75);
	\draw[thick, blue, densely dotted] (1, 3.5) -- (0.75, 3.25) -- (1, 3);
	\draw[thick, blue, densely dotted] (1, 2.5) -- (0.75, 2.25) -- (1, 2);
	\draw[thick, blue, densely dotted] (1, 1.5) -- (0.75, 1.25) -- (1, 1);
	\draw[thick, red] (1,1.5) -- (0,1.5) -- (0,3.5) -- (1, 3.5);
	\draw[thick] (1,2) -- (0.5,2) -- (0.5,3) -- (1, 3);
	\node[right] at (1,3.5) {$5_r$};
	\draw[red, fill=red] (1,3.5) circle (0.05);
	\node[right] at (1,3) {$6_r$};
	\draw[fill=black] (1,3) circle (0.05);
	\node[right] at (1,2.5) {$7_r$};
	\draw[fill=black] (1,2.5) circle (0.05);
	\node[right] at (1,2) {$8_r$};
	\draw[fill=black] (1,2) circle (0.05);
	\node[right] at (1,1.5) {$9_r$};
	\draw[red, fill=red] (1,1.5) circle (0.05);
	\node[right] at (1,1) {$10_r$};
	\draw[ggreen, fill=ggreen] (1,1) circle (0.05);
	\end{tikzpicture}
\end{align*}
Consequently, by writing $\kappa_\pi$ as a product of cumulants, using linearity of $\kappa_\pi$, and using the fact that $(a_1, b_1)$ and $(a_2, b_2)$ are bi-free (and implicitly using $\varphi(b_2) = 1$), we obtain
\[
\kappa_\pi(\underbrace{a_1+a_2, \ldots, a_1+a_2}_n, \underbrace{b_1, b_2, b_1, b_2, \ldots, b_1, b_2}_{b_1 \text{ occurs }m \text{ times}}) z^n w^m = \kappa_{n, s}(a_2, b_2)z^n \prod^s_{q=1} g_2(0_{j_q}, \pi'_q) g_1(0_{j_q}, K(\pi'_q)) w^{j_q}.
\]
Consequently, summing over all $\rho \in BNC_T(n,m)_e$ with $V_\rho = V_\pi$, we obtain 
\[
\sum_{\substack{\rho \in BNC_T(n,m)_e \\ V_\rho = V_\pi  }} \kappa_\rho(\underbrace{a_1+a_2, \ldots, a_1+a_2}_n, \underbrace{b_1, b_2, b_1, b_2, \ldots, b_1, b_2}_{b_1 \text{ occurs }m \text{ times}}) z^n w^m = \kappa_{n, s}(a_2, b_2)z^n \prod^s_{q=1} (g_2 \check{\ast} g_1)(0_{j_q}, 1_{j_q}) w^{j_q}.
\]

Finally, if we sum over all possible $n,m\geq 1$ and all possible $V_\pi$ (so, in the above equation, we get all possible $s\geq 1$ and all possible $j_q \geq 1$), we obtain that 
\begin{align*}
\Psi_e(z,w) = \sum_{n,s \geq 1} \kappa_{n, s}(a_2, b_2)z^n \prod^s_{q=1} \phi_{g_2 \check{\ast} g_1}(w) =\sum_{n,s \geq 1} \kappa_{n, s}(a_2, b_2)z^n \left(\phi_{g_2 \check{\ast} g_1}(w)\right)^s = K_{a_2, b_2}\left(z, \phi_{g_2 \check{\ast} g_1}(w) \right)
\end{align*}
as desired.
\end{proof}

In order to discuss $\Psi_o(z,w)$, it will be quite helpful to discuss a subcase.  For $n,m \geq 0$, let $\sigma'_{n,m} $ denote the element of $BNC(n, 2m+1)$ with blocks $\{\{k_\ell\}\}_{k=1}^n \cup \{1_r\} \cup \{\{(2k)_r, (2k+1)_r\}\}_{k=1}^m$.  Let $BNC_T(n,m)'_o$ denote the set of all $\pi \in BNC(n,2m+1)$ such that $\pi \vee \sigma'_{n,m} = 1_{n,2m+1}$ and $\pi$ contains no blocks containing both a $(2k)_r$ and a $(2j-1)_r$ for some $k, j$.

\begin{lem}
\label{lem:T-case-2}
Under the above notation and assumptions, if
\begin{align*}
\Psi_{o'}(z,w) := \sum_{\substack{n\geq 1 \\ m\geq 0}} \left(\sum_{\pi \in BNC_T(n,m)'_o} \kappa_\pi(\underbrace{a_1+a_2, \ldots, a_1+a_2}_n, \underbrace{b_2, b_1, b_2, b_1, \ldots, b_1, b_2}_{b_1 \text{ occurs }m \text{ times}})   \right) z^n w^{m+1},
\end{align*}
then
\[
\Psi_{o'}(z,w) = \frac{w }{\phi_{g_2 \check{\ast} g_1}(w)} K_{a_2, b_2}\left(z, \phi_{g_2 \check{\ast} g_1}(w) \right).
\]
\end{lem}
\begin{proof}
For each $n,m\geq 1$, we desire to rearrange the sum in $\Psi_{o'}(z,w)$ by expanding $\kappa_\pi$ as a product of full $(\ell, r)$-cumulants and summing over all $\pi$ with the same block containing $1_\ell$.  

Fix $n \geq 1$ and $m \geq 0$.  If $\pi \in BNC_T(n,m)'_o$, then the block $V_\pi$ containing $1_\ell$ must contain $1_r, (2m+1)_r, 1_\ell, 2_\ell, \ldots, n_\ell$ in order for $\pi \vee \sigma'_{n,m} = 1_{n,2m+1}$.  Below is an example of such a $\pi$. 
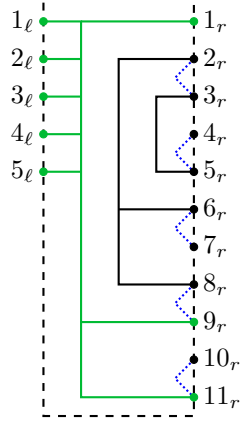
\begin{align*}
	\begin{tikzpicture}[baseline]
	\draw[thick, dashed] (-1,5.75) -- (-1,.25) -- (1,.25) -- (1,5.75);
	\draw[thick, blue, densely dotted] (1, 5) -- (0.75, 4.75) -- (1, 4.5);
	\draw[thick, blue, densely dotted] (1, 4) -- (0.75, 3.75) -- (1, 3.5);
	\draw[thick, blue, densely dotted] (1, 3) -- (0.75, 2.75) -- (1, 2.5);
	\draw[thick, blue, densely dotted] (1, 2) -- (0.75, 1.75) -- (1, 1.5);
	\draw[thick, blue, densely dotted] (1, 1) -- (0.75, 0.75) -- (1, 0.5);
	\draw[ggreen, thick] (-1,5.5) -- (-0.5,5.5) -- (-.5,0.5) -- (1, 0.5);
	\draw[thick] (1,4.5) -- (0.5,4.5) -- (.5,3.5) -- (1, 3.5);
	\draw[thick] (1,5) -- (-0,5) -- (-0,2) -- (1, 2);
	\draw[thick] (1,3) -- (-0,3);
	\draw[ggreen, thick] (-1, 5) -- (-0.5, 5);
	\draw[ggreen, thick] (-1, 4.5) -- (-0.5, 4.5);
	\draw[ggreen, thick] (-1, 4) -- (-0.5, 4);
	\draw[ggreen, thick] (-1, 3.5) -- (-0.5, 3.5);
	\draw[ggreen, thick] (1, 1.5) -- (-0.5, 1.5);
	\draw[ggreen, thick] (1, 5.5) -- (-0.5, 5.5);
	\node[left] at (-1, 5.5) {$1_\ell$};
	\draw[ggreen, fill=ggreen] (-1,5.5) circle (0.05);
	\node[left] at (-1, 5) {$2_\ell$};
	\draw[ggreen, fill=ggreen] (-1,5) circle (0.05);
	\node[left] at (-1, 4.5) {$3_\ell$};
	\draw[ggreen, fill=ggreen] (-1,4.5) circle (0.05);
	\node[left] at (-1, 4) {$4_\ell$};
	\draw[ggreen, fill=ggreen] (-1,4) circle (0.05);
	\node[left] at (-1, 3.5) {$5_\ell$};
	\draw[ggreen, fill=ggreen] (-1,3.5) circle (0.05);
	\draw[ggreen, fill=ggreen] (1,5.5) circle (0.05);
	\node[right] at (1,5.5) {$1_r$};
	\draw[fill=black] (1,5) circle (0.05);
	\node[right] at (1,5) {$2_r$};
	\draw[fill=black] (1,4.5) circle (0.05);
	\node[right] at (1,4.5) {$3_r$};
	\node[right] at (1,4) {$4_r$};
	\draw[fill=black] (1,4) circle (0.05);
	\node[right] at (1,3.5) {$5_r$};
	\draw[fill=black] (1,3.5) circle (0.05);
	\node[right] at (1,3) {$6_r$};
	\draw[fill=black] (1,3) circle (0.05);
	\node[right] at (1,2.5) {$7_r$};
	\draw[fill=black] (1,2.5) circle (0.05);
	\node[right] at (1,2) {$8_r$};
	\draw[fill=black] (1,2) circle (0.05);
	\node[right] at (1,1.5) {$9_r$};
	\draw[ggreen, fill=ggreen] (1,1.5) circle (0.05);
	\node[right] at (1,1) {$10_r$};
	\draw[fill=black] (1,1) circle (0.05);
	\node[right] at (1,0.5) {$11_r$};
	\draw[ggreen, fill=ggreen] (1,0.5) circle (0.05);
	\end{tikzpicture}
\end{align*}

Let $E = \{(2k)_r\}^{m}_{k=1}$, let $O = \{(2k-1)_r\}^{m+1}_{k=1}$, let $s$ denote the number of elements of $O$ contained in $V_\pi$ (so $s \geq 1$), and let $1 = k_1 < k_2 < \cdots < k_s = m+1$ be such that $(2k_q-1)_r \in V_\pi$. Note $V_\pi$ divides the right nodes into $s-1$ disjoint regions.  For each $1 \leq q \leq s-1$, let $j_q = k_{q+1} - k_{q}$ and let $\pi_q$ denote the non-crossing partition obtained by restricting $\pi$ to $\{(2k_{q} )_r, (2k_{q}+1)_r, \ldots, (2k_{q+1}-2)_r\}$.  Note that $\sum^{s-1}_{q=1} j_q = m$.  Furthermore, if $\pi'_q$ is obtained from $\pi_q$ by adding the singleton block $\{(2k_q-1)_r\}$, then $\pi'_q|_{O}$ is naturally an element of $NC'(j_q)$ and $\pi'_q|_E$ is naturally an element of $NC(j_q)$, which must be $K(\pi'_q|_O)$ in order for $\pi \vee \sigma'_{n,m} = 1_{n,2m+1}$.  Consequently, by writing $\kappa_\pi$ as a product of cumulants, using linearity of $\kappa_\pi$, and using the fact that $(a_1, b_1)$ and $(a_2, b_2)$ are bi-free (and implicitly using $\varphi(b_2) = 1$), we obtain
\[
\kappa_\pi(\underbrace{a_1+a_2, \ldots, a_1+a_2}_n, \underbrace{b_2, b_1, b_2, b_1, \ldots, b_1, b_2}_{b_1 \text{ occurs }m \text{ times}}) z^n w^{m+1}  = \kappa_{n, s}(a_2, b_2)z^n w \prod^{s-1}_{q=1} g_2(0_{j_q}, \pi'_q) g_1(0_{j_q}, K(\pi'_q)) w^{j_q}.
\]
Consequently, summing over all $\rho \in BNC_T(n,m)'_o$ with $V_\rho = V_\pi$, we obtain 
\begin{align*}
\sum_{\substack{\rho \in BNC_T(n,m)'_o \\ V_\rho = V_\pi  }}\kappa_\pi &(\underbrace{a_1+a_2, \ldots, a_1+a_2}_n, \underbrace{b_2, b_1, b_2, b_1, \ldots, b_1, b_2}_{b_1 \text{ occurs }m \text{ times}}) z^n w^{m+1}\\ &= \kappa_{n, s}(a_2, b_2)z^n w \prod^{s-1}_{q=1} (g_2 \check{\ast} g_1)(0_{j_q}, 1_{j_q}) w^{j_q}.
\end{align*}

Finally, if we sum over all possible $n \geq 1$, $m\geq 0$, and all possible $V_\pi$ (so, in the above equation, we get all possible $s\geq 1$ and all possible $j_q\geq 1$), we obtain that
\begin{align*}
\Psi_{o'}(z,w) &= \sum_{n,s \geq 1} \kappa_{n, s}(a_2, b_2)z^n w \prod^{s-1}_{q=1} \phi_{g_2 \check{\ast} g_1}(w) \\
&= \frac{w}{ \phi_{g_2 \check{\ast} g_1}(w) }\sum_{n,s \geq 1} \kappa_{n, s}(a_2, b_2)z^n (\phi_{g_2 \check{\ast} g_1}(w))^s \\
&= \frac{w }{\phi_{g_2 \check{\ast} g_1}(w)} K_{a_2, b_2}\left(z, \phi_{g_2 \check{\ast} g_1}(w) \right). \qedhere
\end{align*}
\end{proof}

\begin{lem}
\label{lem:T-case-3}
Under the above notation and assumptions, 
\[
\Psi_o(z,w) =  \left(1 + \frac{1}{\phi_{g_1 \check{\ast} g_2}(w)}\Psi_{o'}(z,w) \right)  K_{a_1, b_1}(z, \phi_{g_1 \check{\ast} g_2}(w)). 
\]
\end{lem}
\begin{proof}
For each $n,m\geq 1$, we desire to rearrange the sum in $\Psi_o(z,w)$ by expanding $\kappa_\pi$ as a product of full $(\ell, r)$-cumulants and summing over all $\pi$ with the same block containing $1_\ell$. 

Fix $n,m \geq 1$, let $E = \{(2k)_r\}^{m}_{k=1}$, let $O = \{(2k-1)_r\}^{m}_{k=1}$, let $\pi \in BNC_T(n,m)_o$, let $V_\pi$ denote the block of $\pi$ containing $1_\ell$, let $t$ (respectively $s$) denote the number of elements of $\{1_\ell, \ldots, n_\ell\}$ (respectively $O$) contained in $V_\pi$ (so $t, s \geq 1$).  Since $\pi \vee \sigma_{n,m} = 1_{n,2m}$, $V_\pi$ must be of the form $\{k_\ell\}^t_{k=1} \cup \{(2k_q-1)_r\}^s_{q=1}$ for some $1 = k_1 < k_2 < \cdots < k_s \leq m$.  Below is an example of such a $\pi$.
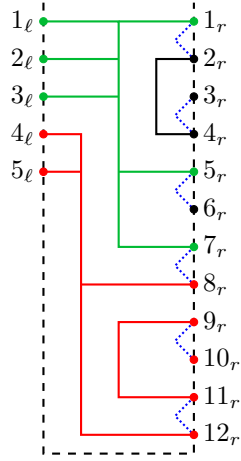
\begin{align*}
	\begin{tikzpicture}[baseline]
	\draw[thick, dashed] (-1,5.75) -- (-1,-.25) -- (1,-.25) -- (1,5.75);
	\draw[thick, blue, densely dotted] (1, 5.5) -- (0.75, 5.25) -- (1, 5);
	\draw[thick, blue, densely dotted] (1, 4.5) -- (0.75, 4.25) -- (1, 4);
	\draw[thick, blue, densely dotted] (1, 3.5) -- (0.75, 3.25) -- (1, 3);
	\draw[thick, blue, densely dotted] (1, 2.5) -- (0.75, 2.25) -- (1, 2);
	\draw[thick, blue, densely dotted] (1, 1.5) -- (0.75, 1.25) -- (1, 1);
	\draw[thick, blue, densely dotted] (1, 0.5) -- (0.75, 0.25) -- (1, 0);
	\draw[ggreen, thick] (-1,5.5) -- (0, 5.5) -- (0, 2.5) -- (1, 2.5);
	\draw[thick] (1,5) -- (.5, 5) -- (.5, 4) -- (1, 4);
	\draw[thick, red] (1,.5) -- (0, .5) -- (0, 1.5) -- (1, 1.5);
	\draw[thick, red] (-1,4) -- (-.5, 4) -- (-.5, 0) -- (1, 0);
	\draw[ggreen, thick] (1, 5.5) -- (0, 5.5);
	\draw[ggreen, thick] (1, 3.5) -- (0, 3.5);
	\draw[ggreen, thick] (-1, 4.5) -- (0, 4.5);
	\draw[ggreen, thick] (-1, 5) -- (0, 5);
	\draw[thick, red] (1, 2) -- (-.5, 2);
	\draw[thick, red] (-1, 3.5) -- (-.5, 3.5);
	\node[left] at (-1, 5.5) {$1_\ell$};
	\draw[ggreen, fill=ggreen] (-1,5.5) circle (0.05);
	\node[left] at (-1, 5) {$2_\ell$};
	\draw[ggreen, fill=ggreen](-1,5) circle (0.05);
	\node[left] at (-1, 4.5) {$3_\ell$};
	\draw[ggreen, fill=ggreen] (-1,4.5) circle (0.05);
	\node[left] at (-1, 4) {$4_\ell$};
	\draw[red, fill=red] (-1,4) circle (0.05);
	\node[left] at (-1, 3.5) {$5_\ell$};
	\draw[red, fill=red] (-1,3.5) circle (0.05);
	\draw[ggreen, fill=ggreen] (1,5.5) circle (0.05);
	\node[right] at (1,5.5) {$1_r$};
	\draw[fill=black] (1,5) circle (0.05);
	\node[right] at (1,5) {$2_r$};
	\draw[fill=black] (1,4.5) circle (0.05);
	\node[right] at (1,4.5) {$3_r$};
	\draw[red, fill=red] (1,0) circle (0.05);
	\node[right] at (1,4) {$4_r$};
	\draw[fill=black] (1,4) circle (0.05);
	\node[right] at (1,3.5) {$5_r$};
	\draw[ggreen, fill=ggreen] (1,3.5) circle (0.05);
	\node[right] at (1,3) {$6_r$};
	\draw[fill=black] (1,3) circle (0.05);
	\node[right] at (1,2.5) {$7_r$};
	\draw[ggreen, fill=ggreen] (1,2.5) circle (0.05);
	\node[right] at (1,2) {$8_r$};
	\draw[red, fill=red] (1,2) circle (0.05);
	\node[right] at (1,1.5) {$9_r$};
	\draw[red, fill=red] (1,1.5) circle (0.05);
	\node[right] at (1,1) {$10_r$};
	\draw[red, fill=red] (1,1) circle (0.05);
	\node[right] at (1,0.5) {$11_r$};
	\draw[red, fill=red] (1,0.5) circle (0.05);
	\node[right] at (1,0) {$12_r$};
	\end{tikzpicture}
\end{align*}

Note $V_\pi$ divides the right nodes into $s$ disjoint regions where the bottom region is special as those nodes may connect to left nodes.  For each $1 \leq q \leq s$, let $j_q = k_{q+1} - k_{q}$, where $k_s = m+1$.  Note that $\sum^s_{q=1} j_q = m$.   For $q \neq s$, let $\pi_q$ denote the non-crossing partition obtained by restricting $\pi$ to $\{(2k_{q})_r, (2k_{q}+1)_r, \ldots, (2k_{q+1}-2)_r\}$.   As discussed in Lemma \ref{lem:T-case-1}, if $\pi'_q$ is obtained from $\pi_q$ by adding the singleton block $\{(2k_q-1)_r\}$, then $\pi'_q|_{O}$ is naturally an element of $NC'(j_q)$ and $\pi'_q|_E$ is naturally an element of $NC(j_q)$, which must be $K(\pi'_q|_O)$ in order for $\pi \vee \sigma_{n,m} = 1_{n,2m}$. 

Let $\pi'_s$ denote the bi-non-crossing partition obtained by restricting $\pi$ to $\{k_\ell\}_{k=t+1}^n \cup \{(2k_s)_r, (2k_s + 1)_r, \ldots, (2m)_r\}$ (which is shaded differently in the above diagram).  Notice, in order for $\pi \vee \sigma_{n,m} = 1_{2n,2m}$, it must be the case that $\pi_s \in BNC_T(n-t,j_s-1)'_o$.

By writing $\kappa_\pi$ as a product of cumulants, using linearity of $\kappa_\pi$, and using the fact that $(a_1, b_1)$ and $(a_2, b_2)$ are bi-free (and implicitly using $\varphi(b_1) = 1$), we obtain
\begin{align*}
\kappa_\pi&(\underbrace{a_1+a_2, \ldots, a_1+a_2}_n, \underbrace{b_1, b_2, b_1, b_2, \ldots, b_1, b_2}_{b_1 \text{ occurs }m \text{ times}}) z^n w^m \\
& = \kappa_{t, s}(a_1, b_1)z^t \left(\prod^{s-1}_{q=1} g_1(0_{j_q}, \pi'_q) g_2(0_{j_q}, K(\pi'_q)) w^{j_q}\right) \kappa_{\pi_s}(\underbrace{a_1+a_2, \ldots, a_1+a_2}_{n-t}, \underbrace{b_2, b_1, b_2, \ldots, b_1, b_2}_{b_2 \text{ occurs }j_s \text{ times}})z^{n-t} w^{j_s}.
\end{align*}
Consequently, summing over all $\rho \in BNC_T(n,m)_o$ with $V_\rho = V_\pi$, we obtain 
\begin{align*}
\sum_{\substack{\rho \in BNC_T(n,m)_o \\ V_\rho = V_\pi  }} &\kappa_\rho(\underbrace{a_1+a_2, \ldots, a_1+a_2}_n, \underbrace{b_1, b_2, b_1, b_2, \ldots, b_1, b_2}_{b_1 \text{ occurs }m \text{ times}}) z^n w^m \\
= \kappa_{t, s}&(a_1, b_1)z^t \left(\prod^{s-1}_{q=1} (g_1 \check{\ast} g_2)(0_{j_q}, 1_{j_q}) w^{j_q} \right)\\
& \cdot \left( \sum_{\sigma \in BNC_T(n-t,j_s-1)'_o}  \kappa_{\sigma}(\underbrace{a_1+a_2, \ldots, a_1+a_2}_{n-t}, \underbrace{b_2, b_1, b_2, \ldots, b_1, b_2}_{b_2 \text{ occurs }j_s \text{ times}})z^{n-t} w^{j_s} \right)
\end{align*}
as all $\sigma \in BNC_T(n-t,j_s-1)'_o$ occur.

We desire to sum over all $n,m\geq 1$ and all possible $V_\pi$.  This produces all possible $t,s \geq 1$ and all $j_q \geq 1$.  If we first sum those above terms with $n = t$, we see, using similar arguments to those used above, that
\begin{align*}
\sum_{\sigma \in BNC_T(0,j_s-1)'_o}  \kappa_{\sigma}(\underbrace{b_2, b_1, b_2, \ldots, b_1, b_2}_{b_2 \text{ occurs }j_q \text{ times}}) w^{j_s} = (g_1 \check{\ast} g_2)(0_{j_s}, 1_{j_s}) w^{j_s}.
\end{align*}
Consequently, summing those terms with $t = n$ gives
\begin{align*}
\sum_{t,s \geq 1} \kappa_{t, s}(a_1, b_1)z^t \prod^s_{q=1} \phi_{g_1 \check{\ast} g_2}(w) = \sum_{t,s \geq 1} \kappa_{t, s}(a_1, b_1)z^t (\phi_{g_1 \check{\ast} g_2}(w) )^s = K_{a_1, b_1}\left(z, \phi_{g_1 \check{\ast} g_2}(w) \right).
\end{align*}
Moreover, summing those terms with $t \neq n$ gives
\begin{align*}
\sum_{t,s \geq 1} \kappa_{t, s}(a_1, b_1)z^t \left(\prod^{s-1}_{q=1} \phi_{g_1 \check{\ast} g_2}(w)\right) \Psi_{o'}(z,w) &= \sum_{t,s \geq 1} \kappa_{t, s}(a_1, b_1)z^t \left(\phi_{g_1 \check{\ast} g_2}(w)\right)^{s-1} \Psi_{o'}(z,w) \\
&= \frac{1}{\phi_{g_1 \check{\ast} g_2}(w)}\Psi_{o'}(z,w) K_{a_1, b_1}(z, \phi_{g_1 \check{\ast} g_2}(w)).
\end{align*}
Combining the above two sums completes the proof.
\end{proof}

\begin{proof}[Proof of Theorem \ref{thm:T-property}.]
By Lemma \ref{lem:T-case-1} along with equation (\ref{eq:inversion-with-convolution}), we see that
\begin{align*}
\Psi_e\left(z, \phi^{\inv}_{g}(w)\right) &= K_{a_2, b_2}\left(z, \phi_{g_2 \check{\ast} g_1}\left(\phi^{\inv}_{g}(w)\right) \right)  = K_{a_2, b_2}\left(z,\phi^{\inv}_{g_2}(w) \right).
\end{align*}
By Lemma \ref{lem:T-case-2} along with equations (\ref{eq:inversion-with-convolution}, \ref{eq:convolution-with-inverse-series}), we see that
\begin{align*}
\Psi_{o'}\left(z,\phi^{\inv}_{g}(w)\right) &= \frac{\phi^{\inv}_{g}(w) }{\phi_{g_2 \check{\ast} g_1}\left(\phi^{\inv}_{g}(w)\right)} K_{a_2, b_2}\left(z, \phi_{g_2 \check{\ast} g_1}\left(\phi^{\inv}_{g}(w)\right) \right)\\
&= \frac{\frac{1}{w}\phi^{\inv}_{g_1}(w)\phi^{\inv}_{g_2}(w)}{\phi^{\inv}_{g_2}(w)} K_{a_2, b_2}\left(z, \phi^{\inv}_{g_2}(w) \right)\\
&=\frac{1}{w} \phi^{\inv}_{g_1}(w)  K_{a_2, b_2}\left(z, \phi^{\inv}_{g_2}(w) \right).
\end{align*}
Furthermore, by Lemma \ref{lem:T-case-3} along with equation (\ref{eq:inversion-with-convolution}), we obtain
\begin{align*}
\Psi_o\left(z,\phi^{\inv}_{g}(w)\right) &= \left(1 + \frac{1}{\phi_{g_1 \check{\ast} g_2}\left(\phi^{\inv}_{g}(w)\right)}\Psi_{o'}\left(z,\phi^{\inv}_{g}(w)\right) \right)  K_{a_1, b_1}\left(z, \phi_{g_1 \check{\ast} g_2}\left(\phi^{\inv}_{g}(w)\right)\right)\\
&= \left(1 + \frac{1}{\phi^{\inv}_{g_1}(w)}\Psi_{o'}\left(z,\phi^{\inv}_{g}(w)\right)\right) K_{a_1, b_1}\left(z, \phi^{\inv}_{g_1}(w)\right) \\
&= \left(1 + \frac{1}{w}  K_{a_2, b_2}\left(z, \phi^{\inv}_{g_2}(w) \right)\right) K_{a_1, b_1}\left(z, \phi^{\inv}_{g_1}(w)\right) \\
&= K_{a_1, b_1}\left(z, \phi^{\inv}_{g_1}(w)\right) + \frac{1}{w} K_{a_1, b_1}\left(z, \phi^{\inv}_{g_1}(w)\right) K_{a_2, b_2}\left(z, \phi^{\inv}_{g_2}(w) \right).
\end{align*}
As
\begin{align*}
K_{a_1+a_2, b_1b_2}\left(z,\phi^{\inv}_{g}(w)\right) &= \Psi_e\left(z, \phi^{\inv}_{g}(w)\right)  + \Psi_o\left(z,\phi^{\inv}_{g}(w)\right),
\end{align*}
we have verified equation (\ref{eq:T-eq-to-verify}) holds and thus the proof is complete.
\end{proof}

\section{Bi-Free Partial $S$-Transform}
\label{sec:S}

In this section, we will study Voiculescu's bi-free partial $S$-transform through combinatorics.    All notation in this section refers to the notation established in this section and not to the notation of Section \ref{sec:T}.

\begin{defn}[\cite{V2015}*{Definition 2.1}]
Let $(a, b)$ be a two-faced pair in a non-commutative probability space $(\A, \varphi)$ with $\varphi(a) \neq 0$ and $\varphi(b) \neq 0$.  The $2$-variable partial bi-free $S$-transform of $(a, b)$ is the holomorphic function on $(\bC \setminus \{0\})^2$ near $(0,0)$ defined by
\begin{align}
S_{a,b}(z, w) = \frac{z+1}{z}\frac{w+1}{w} \left( 1-\frac{z}{H_{a,b}(\X_a(z), \X_b(w))}\right). \label{eq:S-V}
\end{align}
\end{defn}

It will be useful to note the following equivalent definition of the bi-free partial $S$-transform.  To simplify discussions, we will demonstrate the equality in the case $\varphi(a) = \varphi(b) = 1$.  This does not hinder the proof of the desired result; that is, Theorem \ref{thm:S-property} (see Remark \ref{rem:S-b=1}).

\begin{prop}
\label{prop:S-new-defn}
If $(a,b)$ is a two-faced pair in a non-commutative probability space $(\A, \varphi)$ with $\varphi(a) = \varphi(b) =1$, then, as a formal power series,
\begin{align}
S_{a,b}(z,w) = 1 + \frac{1+z+w}{zw} K_{a,b}\left( c^{\inv}_a(z),  c^{\inv}_b(w)\right). \label{eq:S-my}
\end{align}
\end{prop}
\begin{proof}
Using equations (\ref{eq:X-in-moment}, \ref{eq:S}, \ref{eq:S-to-cumulants}, \ref{eq:bi-moment}), we obtain that
\begin{align*}
\frac{1}{H_{a,b}\left(\X_a(z), \X_b(w)\right)} = \frac{1}{1+z} + \frac{1}{1+w} -\frac{1}{1+z}\frac{1}{1+w}C_{a,b}\left(c^{\inv}_a(z), c^{\inv}_b(w)\right)   
\end{align*}
Therefore, using equations (\ref{eq:bi-K}, \ref{eq:S-V}), we obtain that
\begin{align*}
S_{a,b}(z,w) &= \frac{z+1}{z}\frac{w+1}{w} \left(1 - (1+z+w) \left( \frac{1}{1+z} + \frac{1}{1+w} -\frac{1}{1+z}\frac{1}{1+w}C_{a,b}\left(c^{\inv}_a(z), c^{\inv}_b(w)\right)     \right)\right) \\
&= \frac{1}{zw}\left((1+z)(1+w) - (1+z+w)(2+z+w) + (1+z+w)C_{a,b}\left(c^{\inv}_a(z), c^{\inv}_b(w)\right)         \right)  \\
&= \frac{1}{zw}\left( zw - (1+z+w)^2 + (1+z+w)\left(1 + z + w +  K_{a,b}\left( c^{\inv}_a(z),  c^{\inv}_b(w)\right)    \right)    \right) \\
&=  1 + \frac{1+z+w}{zw} K_{a,b}\left( c^{\inv}_a(z),  c^{\inv}_b(w)\right). \qedhere
\end{align*}
\end{proof}

\begin{rem}
\label{rem:S-b=1}
Again, one might be concerned that we have restricted to the case $\varphi(a) =\varphi(b) = 1$.  Using the same ideas as in Remark \ref{rem:T-b=1}, if we use equation (\ref{eq:S-my}) as the definition of the $S$-transform and if $\lambda, \mu \in \bC \setminus \{0\}$, then $S_{a,b}(z,w) = S_{\lambda a,\mu b}(z,w)$.  Hence there is no loss in assuming $\varphi(a) = \varphi(b) = 1$.
\end{rem}
\begin{rem}
Note Proposition \ref{prop:S-new-defn} immediately provides the $S$-transform part of \cite{V2015}*{Proposition 4.2}.  Indeed if $a$ and $b$ are elements of a non-commutative probability space $(\A, \varphi)$ with $\varphi(a) \neq 0$, $\varphi(b) \neq 0$, and $\varphi(a^nb^m) = \varphi(a^n)\varphi(b^m)$ for all $n,m \geq 0$, then $\kappa_{n,m}(a,b) = 0$ for all $n,m\geq 1$ (see \cite{S2014}*{Section 3.2}).  Hence $K_{a,b}(z,w) = 0$ so $S_{a,b}(z,w) = 1$.
\end{rem}

We desire to prove the following, which was one of two main results of \cite{V2015}, using combinatorics via Proposition \ref{prop:S-new-defn}.  
\begin{thm}[\cite{V2015}*{Theorem 2.1}]
\label{thm:S-property}
Let $(a_1, b_1)$ and $(a_2, b_2)$ be bi-free two-faced pairs in a non-commutative probability space $(\A, \varphi)$ with $\varphi(a_j) \neq 0$ and $\varphi(b_j) \neq 0$.  Then
\[
S_{a_1a_2, b_1b_2}(z,w) = S_{a_1, b_1}(z,w) S_{a_2, b_2}(z,w)
\]
on $(\bC \setminus \{0\})^2$ near $(0,0)$.
\end{thm}

To simplify the proof of the result, we will assume that $\varphi(a_j) = \varphi(b_j) = 1$.  Note $\varphi(a_1a_2) = \varphi(b_1b_2) = 1$ by freeness of the left algebras and of the right algebras in bi-free pairs.  Furthermore, we will let $f_j$ (respectively $g_j$) denote the multiplicative function associated to the cumulants of $a_j$ (respectively $b_j$) defined by $f_j(0_n, 1_n) = \kappa_n(a_j)$ (respectively $g_j(0_n, 1_n) = \kappa_n(b_j)$).  Recall if $f$ (respectively $g$) is the multiplicative function associated to the cumulants of $a_1a_2$ (respectively $b_1b_2$), then $f = f_1 \ast f_2$ (respectively $g = g_1 \ast g_2$).  Thus $\phi^{\inv}_{f}(z) = c^{\inv}_{a_1a_2}(z)$, $\phi^{\inv}_g(w) = c^{\inv}_{b_1b_2}(w)$, $\phi^{\inv}_{f_j}(z) = c^{\inv}_{a_j}(z)$, and $\phi^{\inv}_{g_j}(w) = c^{\inv}_{b_j}(w)$.

By Proposition \ref{prop:S-new-defn}, it suffices to show that
\begin{align}
K_{a_1a_2, b_1b_2}\left(\phi^{\inv}_{f}(w), \phi^{\inv}_{g}(w)\right) =  \Theta_1(z,w) + \Theta_2(z,w) + \frac{1+z+w}{zw} \Theta_1(z,w)\Theta_2(z,w) \label{eq:S-eq-to-verify}
\end{align}
where
\[
\Theta_j(z,w) = K_{a_j, b_j}\left(\phi^{\inv}_{f_j}(w), \phi^{\inv}_{g_j}(w)\right).
\]

Recall
\[
K_{a_1a_2, b_1b_2}(z,w) = \sum_{n,m\geq 1} \kappa_{n,m}(a_1a_2, b_1b_2) z^n w^m.
\]
For fix $n,m \geq 1$, let $\sigma_{n,m}$ denote the element of $BNC(2n,2m)$ with blocks $\{\{(2k-1)_\ell, (2k)_\ell\}\}_{k=1}^n \cup \{\{(2k-1)_r, (2k)_r\}\}_{k=1}^m$.  Thus equation (\ref{eq:product-of-cumulants}) implies that
\[
\kappa_{n,m}(a_1a_2, b_1b_2) = \sum_{\substack{ \pi \in BNC(2n, 2m) \\ \pi \vee \sigma_{n,m} = 1_{2n,2m} }} \kappa_\pi(\underbrace{a_1, a_2, a_1, a_2, \ldots, a_1, a_2}_{a_1 \text{ occurs }n \text{ times}}, \underbrace{b_1, b_2, b_1, b_2, \ldots, b_1, b_2}_{b_1 \text{ occurs }m \text{ times}}).
\]
Since $(a_1, b_1)$ and $(a_2, b_2)$ are bi-free, we note that 
\[
\kappa_\pi(\underbrace{a_1, a_2, a_1, a_2, \ldots, a_1, a_2}_{a_1 \text{ occurs }n \text{ times}}, \underbrace{b_1, b_2, b_1, b_2, \ldots, b_1, b_2}_{b_1 \text{ occurs }m \text{ times}} )= 0
\]
if $\pi$ contains a block containing a $(2k)_{\theta_1}$ and a $(2j-1)_{\theta_2}$ for some $\theta_1, \theta_2 \in \{\ell, r\}$ and for some $k, j$.

For $n,m\geq 1$, let $BNC_S(n,m)$ denote all $\pi \in BNC(2n, 2m)$ such that $\pi \vee \sigma_{n,m} = 1_{2n,2m}$ and $\pi$ contains no blocks with both a $(2k)_{\theta_1}$ and a $(2j-1)_{\theta_2}$ for some $\theta_1, \theta_2 \in \{\ell, r\}$ and for some $k, j$.  Consequently, we obtain
\[
K_{a_1a_2, b_1b_2}(z,w) = \sum_{n,m\geq 1} \left(\sum_{\pi \in BNC_S(n,m)} \kappa_\pi(\underbrace{a_1, a_2, a_1, a_2, \ldots, a_1, a_2}_{a_1 \text{ occurs }n \text{ times}}, \underbrace{b_1, b_2, b_1, b_2, \ldots, b_1, b_2}_{b_1 \text{ occurs }m \text{ times}}) \right) z^n w^m.
\]

We desire to divide up this sum into two parts based on types of partitions in $BNC_S(n,m)$.  Notice that if $\pi \in BNC_S(n,m)$, then $\pi$ must contain a block with both a $k_\ell$ and a $j_r$ for some $k,j$ so that $\pi \vee \sigma_{n,m} = 1_{2n, 2m}$.   If $V \subseteq \{1_\ell, \ldots, (2n)_\ell, 1_r, \ldots, (2m)_r\}$, we define $\min(V)$ to be the integer $k$ such that either $k_\ell \in V$ or $k_r \in V$ yet $j_\ell, j_r \notin V$ for all $j < k$.  

Let $BNC_S(n,m)_e$ denote all $\pi \in BNC_S(n,m)$ such that the block $V$ of $\pi$ that has the smallest $\min$-value over all blocks $W$ of $\pi$ such that there exists $k_\ell, j_r \in W$ for some $k,j$ has $\min(V) \in 2\bZ$; that is, $V$ is the first block, measured from the top, in the bi-non-crossing diagram of $\pi$ that has both left and right nodes and these nodes are of even index.  Similarly let $BNC_S(n,m)_o$ denote all $\pi \in BNC_T(n,m)$ such that the block $V$ of $\pi$ that has the smallest $\min$-value over all blocks $W$ of $\pi$ such that there exists $k_\ell, j_r \in W$ for some $k,j$ has $\min(V) \in 2\bZ+1$.  Note $BNC_S(n,m)_e$ and $BNC_S(n,m)_o$ are disjoint and $BNC_S(n,m)_e \cup BNC_S(n,m)_o = BNC_S(n,m)$.  Therefore, if for $d \in \{o,e\}$ we define
\[
\Psi_d(z,w) := \sum_{n,m\geq 1} \left(\sum_{\pi \in BNC_S(n,m)_d} \kappa_\pi(\underbrace{a_1, a_2, a_1, a_2, \ldots, a_1, a_2}_{a_1 \text{ occurs }n \text{ times}}, \underbrace{b_1, b_2, b_1, b_2, \ldots, b_1, b_2}_{b_1 \text{ occurs }m \text{ times}} )\right) z^n w^m,
\]
then
\[
K_{a_1a_2, b_1b_2}(z,w) = \Psi_e(z,w) + \Psi_o(z,w).
\]
We will derive expressions for $\Psi_e(z,w)$ and $\Psi_o(z,w)$ beginning with $\Psi_e(z,w)$.  We will not use the same rigour as we did in Section \ref{sec:T} as most of the arguments are similar.

\begin{lem}
\label{lem:S-case-1}
Under the above notation and assumptions,
\[
\Psi_e(z,w) = K_{a_2, b_2}\left(\phi_{f_2 \check{\ast} f_1}(z), \phi_{g_2 \check{\ast} g_1}(w) \right). 
\]
\end{lem}
\begin{proof}
Fix $n,m \geq 1$.  If $\pi \in BNC_S(n,m)_e$, let $V_\pi$ denote the first (and, as it happens, only) block of $\pi$, as measured from the top of $\pi$'s bi-non-crossing diagram, that has both left and right nodes.  Since $\pi \vee \sigma_{n,m} = 1_{2n,2m}$, there exist $t,s\geq 1$, $1\leq l_1 < l_2 < \cdots < l_t = n$, and $1 \leq k_1 < k_2 < \cdots < k_s = m$ such that
\[
V_\pi = \{(2l_p)_\ell\}^t_{p=1} \cup \{(2k_q)_r\}^s_{q=1}.
\]
Note $V_\pi$ divides the remaining left nodes into $t$ disjoint regions and the remaining right nodes into $s$ disjoint regions. Moreover, each block of $\pi$ can only contain nodes in one such region.  Below is an example of such a $\pi$.
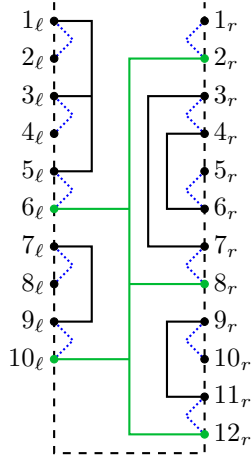
\begin{align*}
	\begin{tikzpicture}[baseline]
	\draw[thick, dashed] (-1,5.75) -- (-1,-.25) -- (1,-.25) -- (1,5.75);
	\draw[thick, blue, densely dotted] (1, 5.5) -- (0.75, 5.25) -- (1, 5);
	\draw[thick, blue, densely dotted] (1, 4.5) -- (0.75, 4.25) -- (1, 4);
	\draw[thick, blue, densely dotted] (1, 3.5) -- (0.75, 3.25) -- (1, 3);
	\draw[thick, blue, densely dotted] (1, 2.5) -- (0.75, 2.25) -- (1, 2);
	\draw[thick, blue, densely dotted] (1, 1.5) -- (0.75, 1.25) -- (1, 1);
	\draw[thick, blue, densely dotted] (1, 0.5) -- (0.75, 0.25) -- (1, 0);
	\draw[thick, blue, densely dotted] (-1, 5.5) -- (-0.75, 5.25) -- (-1, 5);
	\draw[thick, blue, densely dotted] (-1, 4.5) -- (-0.75, 4.25) -- (-1, 4);
	\draw[thick, blue, densely dotted] (-1, 3.5) -- (-0.75, 3.25) -- (-1, 3);
	\draw[thick, blue, densely dotted] (-1, 2.5) -- (-0.75, 2.25) -- (-1, 2);
	\draw[thick, blue, densely dotted] (-1, 1.5) -- (-0.75, 1.25) -- (-1, 1);
	\draw[thick, ggreen] (1,5) -- (0,5) -- (0,0) -- (1,0);
	\draw[thick] (-1,5.5) -- (-.5, 5.5) -- (-.5, 3.5) -- (-1, 3.5);
	\draw[thick] (1, 4.5) -- (.25, 4.5) -- (.25, 2.5) -- (1, 2.5);
	\draw[thick] (1, 4) -- (.5, 4) -- (.5, 3) -- (1, 3);
	\draw[thick] (1,1.5) -- (.5, 1.5) -- (.5, 0.5) -- (1, 0.5);
	\draw[thick] (-1, 1.5) -- (-.5, 1.5) -- (-.5, 2.5) -- (-1, 2.5);
	\draw[thick] (-1, 4.5) -- (-.5, 4.5);
	\draw[thick, ggreen] (-1, 3) -- (0, 3);
	\draw[thick, ggreen] (-1, 1) -- (0, 1);
	\draw[thick, ggreen] (1, 2) -- (0, 2);
	\node[left] at (-1, 5.5) {$1_\ell$};
	\draw[fill=black] (-1,5.5) circle (0.05);
	\node[left] at (-1, 5) {$2_\ell$};
	\draw[fill=black] (-1,5) circle (0.05);
	\node[left] at (-1, 4.5) {$3_\ell$};
	\draw[fill=black] (-1,4.5) circle (0.05);
	\node[left] at (-1, 4) {$4_\ell$};
	\draw[fill=black] (-1,4) circle (0.05);
	\node[left] at (-1, 3.5) {$5_\ell$};
	\draw[fill=black] (-1,3.5) circle (0.05);
	\node[left] at (-1, 3) {$6_\ell$};
	\draw[ggreen, fill=ggreen] (-1,3) circle (0.05);
	\node[left] at (-1, 2.5) {$7_\ell$};
	\draw[fill=black] (-1,2.5) circle (0.05);
	\node[left] at (-1, 2) {$8_\ell$};
	\draw[fill=black] (-1,2) circle (0.05);
	\node[left] at (-1, 1.5) {$9_\ell$};
	\draw[fill=black] (-1,1.5) circle (0.05);
	\node[left] at (-1, 1) {$10_\ell$};
	\draw[ggreen, fill=ggreen] (-1,1) circle (0.05);
	\draw[fill=black] (1,5.5) circle (0.05);
	\node[right] at (1,5.5) {$1_r$};
	\draw[ggreen, fill=ggreen] (1,5) circle (0.05);
	\node[right] at (1,5) {$2_r$};
	\draw[fill=black] (1,4.5) circle (0.05);
	\node[right] at (1,4.5) {$3_r$};
	\draw[ggreen, fill=ggreen] (1,0) circle (0.05);
	\node[right] at (1,4) {$4_r$};
	\draw[fill=black] (1,4) circle (0.05);
	\node[right] at (1,3.5) {$5_r$};
	\draw[fill=black] (1,3.5) circle (0.05);
	\node[right] at (1,3) {$6_r$};
	\draw[fill=black] (1,3) circle (0.05);
	\node[right] at (1,2.5) {$7_r$};
	\draw[fill=black] (1,2.5) circle (0.05);
	\node[right] at (1,2) {$8_r$};
	\draw[ggreen, fill=ggreen] (1,2) circle (0.05);
	\node[right] at (1,1.5) {$9_r$};
	\draw[fill=black] (1,1.5) circle (0.05);
	\node[right] at (1,1) {$10_r$};
	\draw[fill=black] (1,1) circle (0.05);
	\node[right] at (1,0.5) {$11_r$};
	\draw[fill=black] (1,0.5) circle (0.05);
	\node[right] at (1,0) {$12_r$};
	\end{tikzpicture}
\end{align*}

Let $E = \{(2k)_\ell\}^{n}_{k=1} \cup \{(2k)_r\}^{m}_{k=1}$ and let $O = \{(2k-1)_\ell\}^{n}_{k=1} \cup \{(2k-1)_r\}^{m}_{k=1}$.  For each $1 \leq p \leq t$, let $i_p = l_p - l_{p-1}$, where $l_0 = 0$, and let $\pi_{\ell, p}$ denote the non-crossing partition obtained by restricting $\pi$ to $\{(2l_{p-1} + 1)_\ell, (2l_{p-1}+2)_\ell, \ldots, (2l_{p}-1)_\ell\}$.  Note that $\sum^t_{p=1} i_p = n$.  Furthermore, as explained in Lemma \ref{lem:T-case-1}, if $\pi'_{\ell, p}$ is obtained from $\pi_{\ell, p}$ by adding the singleton block $\{(2l_p)_\ell\}$, then $\pi'_{\ell, p}|_{E}$ is naturally an element of $NC'(i_p)$ and $\pi'_{\ell, p}|_O$ is naturally an element of $NC(i_p)$, which must be $K(\pi'_{\ell, p}|_E)$ in order for $\pi \vee \sigma_{n,m} = 1_{2n,2m}$. 

Similarly, for each $1 \leq q \leq s$, let $j_q = k_q - k_{q-1}$, where $k_0 = 0$, and let $\pi_{r, q}$ denote the non-crossing partition obtained by restricting $\pi$ to $\{(2k_{q-1} + 1)_r, (2k_{q-1}+2)_r, \ldots, (2k_{q}-1)_r\}$.  Note that $\sum^s_{q=1} j_q = m$.  Furthermore, as explained in Lemma \ref{lem:T-case-1}, if $\pi'_{r, q}$ is obtained from $\pi_{r, q}$ by adding the singleton block $\{(2k_q)_r\}$, then $\pi'_{r, q}|_{E}$ is naturally an element of $NC'(j_q)$ and $\pi'_{r, q}|_O$ is naturally an element of $NC(j_q)$, which must be $K(\pi'_{r, q}|_E)$ in order for $\pi \vee \sigma_{n,m} = 1_{2n,2m}$. 

Expanding
\[
\kappa_\rho(\underbrace{a_1, a_2, \ldots, a_1, a_2}_{a_1 \text{ occurs }n \text{ times}}, \underbrace{b_1, b_2, \ldots, b_1, b_2}_{b_1 \text{ occurs }m \text{ times}}) z^n w^m
\]
for $\rho \in BNC_S(n,m)_e$ and  summing such terms with $V_\rho = V_\pi$, we obtain 
\[
\kappa_{t, s}(a_2, b_2)\left(\prod^t_{p=1} (f_2 \check{\ast} f_1)(0_{i_p}, 1_{i_p}) z^{i_p} \right) \left(\prod^s_{q=1} (g_2 \check{\ast} g_1)(0_{j_q}, 1_{j_q}) w^{j_q}\right).
\]
Finally, if we sum over all possible $n,m\geq 1$ and all possible $V_\pi$ (so, in the above equation, we get all possible $t, s\geq 1$ and all possible $i_p, j_q\geq 1$), we obtain that 
\begin{align*}
\Psi_e(z,w) &= \sum_{t,s \geq 1} \kappa_{t, s}(a_2, b_2)\left(\prod^t_{p=1} \phi_{f_2 \check{\ast} f_1}(z) \right) \left(\prod^s_{q=1} \phi_{g_2 \check{\ast} g_1}(z)\right)  \\
& = \sum_{t,s \geq 1} \kappa_{t, s}(a_2, b_2) \left(\phi_{f_2 \check{\ast} f_1}(z)\right)^t \left(\phi_{g_2 \check{\ast} g_1}(w)\right)^s = K_{a_2, b_2}\left(\phi_{f_2 \check{\ast} f_1}(z), \phi_{g_2 \check{\ast} g_1}(w) \right). \qedhere
\end{align*}
\end{proof}

In order to discuss $\Psi_o(z,w)$, it will be quite helpful to discuss subcases.  For $n,m \geq 0$, let $\sigma'_{n,m} $ denote the element of $BNC(2n+1, 2m+1)$ with blocks $\{\{1_\ell, 1_r\}\} \cup \{\{(2l)_\ell, (2l+1)_\ell\}\}_{l=1}^n  \cup \{\{(2k)_r, (2k+1)_r\}\}_{k=1}^m$.  Let $BNC_S(n,m)'_o$ denote the set of all $\pi \in BNC(2n+1,2m+1)$ such that $\pi \vee \sigma'_{n,m} = 1_{2n+1,2m+1}$ and contains no blocks with both a $(2k)_{\theta_1}$ and a $(2j-1)_{\theta_2}$ any $\theta_1, \theta_2 \in \{\ell, r\}$ and any $k, j$.  We desired to divide up $BNC_S(n,m)'_o$ further.  For $\pi \in BNC_S(n,m)'_o$, let $V_{\pi, \ell}$ denote the block of $\pi$ containing $1_\ell$ and let $V_{\pi, r}$ denote the block of $\pi$ containing $1_r$.  Then
\begin{align*}
BNC_S(n,m)_{o,0} &= \{\pi \in BNC_S(n,m)'_o \, \mid \,  V_{\pi, \ell} \text{ has no right nodes and } V_{\pi, r} \text{ has no left nodes}\},\\
BNC_S(n,m)_{o,r} &= \{\pi \in BNC_S(n,m)'_o \, \mid \,  V_{\pi, \ell} \text{ has no right nodes but } V_{\pi, r} \text{ has left nodes}\}, \\
BNC_S(n,m)_{o,\ell} &= \{\pi \in BNC_S(n,m)'_o \, \mid \,  V_{\pi, \ell} \text{ has right nodes but } V_{\pi, r} \text{ has no left nodes}\} \text{ and} \\
BNC_S(n,m)_{o, \ell r} &= \{\pi \in BNC_S(n,m)'_o \, \mid \, V_{\pi, \ell} = V_{\pi, r}\}.
\end{align*}
Due to the nature of bi-non-crossing partitions, the above sets are disjoint and have union $BNC_S(n,m)'_o$.  

For $d \in \{0, r, \ell, \ell r\}$, define
\begin{align*}
\Psi_{o, d}(z,w) := \sum_{n,m\geq 0} \left(\sum_{\pi \in BNC_S(n,m)_{o,d}} \kappa_\pi(\underbrace{a_2, a_1, a_2, a_1, \ldots, a_1, a_2}_{a_1 \text{ occurs }n \text{ times}}, \underbrace{b_2, b_1, b_2, b_1, \ldots, b_1, b_2}_{b_1 \text{ occurs }m \text{ times}})   \right) z^{n+1} w^{m+1}.
\end{align*}

\begin{lem}
\label{lem:S-case-2}
Under the above notation and assumptions, 
\[
\Psi_{o,0}(z,w) = zw\cdot \frac{\phi_{f_2}\left( \phi_{f_2 \check{\ast} f_1}(z)\right)\phi_{g_2}\left( \phi_{g_2 \check{\ast} g_1}(w)\right)}{\phi_{f_2 \check{\ast} f_1}(z) \phi_{g_2 \check{\ast} g_1}(w)}.
\]
\end{lem}
\begin{proof}
Fix $n,m \geq 0$.  If $\pi \in BNC_S(n,m)_{o,0}$, then, since $\pi \vee \sigma'_{n,m} = 1_{2n+1,2m+1}$, there exist $t,s\geq 1$, $1 = l_1 < l_2 < \cdots < l_t = n+1$, and $1 = k_1 < k_2 < \cdots < k_s = m+1$ such that
\[
V_{\pi, \ell} = \{(2l_p-1)_\ell\}^t_{p=1} \qqand  V_{\pi, r} = \{(2k_q-1)_r\}^s_{q=1}.
\]
Note $V_{\pi, \ell}$ divides the remaining left nodes into $t-1$ disjoint regions and $V_{\pi, r}$ divides the remaining right nodes into $s-1$ disjoint regions. Moreover, each block of $\pi$ can only contain nodes in one such region.  The following is an example of such a $\pi$.
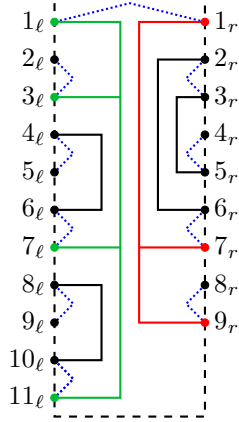
\begin{align*}
	\begin{tikzpicture}[baseline]
	\draw[thick, dashed] (-1,5.75) -- (-1,.25) -- (1,.25) -- (1,5.75);
	\draw[thick, blue, densely dotted] (1, 5.5) -- (0, 5.75) -- (-1, 5.5);
	\draw[thick, blue, densely dotted] (1, 5) -- (0.75, 4.75) -- (1, 4.5);
	\draw[thick, blue, densely dotted] (1, 4) -- (0.75, 3.75) -- (1, 3.5);
	\draw[thick, blue, densely dotted] (1, 3) -- (0.75, 2.75) -- (1, 2.5);
	\draw[thick, blue, densely dotted] (1, 2) -- (0.75, 1.75) -- (1, 1.5);
	\draw[thick, blue, densely dotted] (-1, 1) -- (-0.75, 0.75) -- (-1, 0.5);
	\draw[thick, blue, densely dotted] (-1, 5) -- (-0.75, 4.75) -- (-1, 4.5);
	\draw[thick, blue, densely dotted] (-1, 4) -- (-0.75, 3.75) -- (-1, 3.5);
	\draw[thick, blue, densely dotted] (-1, 3) -- (-0.75, 2.75) -- (-1, 2.5);
	\draw[thick, blue, densely dotted] (-1, 2) -- (-0.75, 1.75) -- (-1, 1.5);
	\draw[thick, blue, densely dotted] (-1, 1) -- (-0.75, 0.75) -- (-1, 0.5);
	\draw[thick] (1, 5) -- (.375, 5) -- (.375, 3) -- (1, 3);	
	\draw[thick] (1, 4.5) -- (.625, 4.5) -- (.625, 3.5) -- (1, 3.5);
	\draw[thick, red] (1, 5.5) -- (.125, 5.5) -- (.125, 1.5) -- (1, 1.5);
	\draw[thick, ggreen] (-1, 5.5) -- (-.125, 5.5) -- (-.125, 0.5) -- (-1, 0.5);
	\draw[thick, ggreen] (-1, 4.5) -- (-.125, 4.5);
	\draw[thick, ggreen] (-1, 2.5) -- (-.125, 2.5);
	\draw[thick, red] (1, 2.5) -- (.125, 2.5);
	\draw[thick] (-1, 4) -- (-.375, 4) -- (-.375, 3) -- (-1, 3);
	\draw[thick] (-1, 2) -- (-.375, 2) -- (-.375, 1) -- (-1, 1);
	\node[left] at (-1, 5.5) {$1_\ell$};
	\draw[ggreen, fill=ggreen] (-1,5.5) circle (0.05);
	\node[left] at (-1, 5) {$2_\ell$};
	\draw[fill=black] (-1,5) circle (0.05);
	\node[left] at (-1, 4.5) {$3_\ell$};
	\draw[ggreen, fill=ggreen] (-1,4.5) circle (0.05);
	\node[left] at (-1, 4) {$4_\ell$};
	\draw[fill=black] (-1,4) circle (0.05);
	\node[left] at (-1, 3.5) {$5_\ell$};
	\draw[fill=black] (-1,3.5) circle (0.05);
	\node[left] at (-1, 3) {$6_\ell$};
	\draw[fill=black] (-1,3) circle (0.05);
	\node[left] at (-1, 2.5) {$7_\ell$};
	\draw[ggreen, fill=ggreen] (-1,2.5) circle (0.05);
	\node[left] at (-1, 2) {$8_\ell$};
	\draw[fill=black] (-1,2) circle (0.05);
	\node[left] at (-1, 1.5) {$9_\ell$};
	\draw[fill=black] (-1,1.5) circle (0.05);
	\node[left] at (-1, 1) {$10_\ell$};
	\draw[fill=black] (-1,1) circle (0.05);
	\node[left] at (-1, .5) {$11_\ell$};
	\draw[ggreen, fill=ggreen] (-1,.5) circle (0.05);
	\draw[red, fill=red] (1,5.5) circle (0.05);
	\node[right] at (1,5.5) {$1_r$};
	\draw[fill=black] (1,5) circle (0.05);
	\node[right] at (1,5) {$2_r$};
	\draw[fill=black] (1,4.5) circle (0.05);
	\node[right] at (1,4.5) {$3_r$};
	\node[right] at (1,4) {$4_r$};
	\draw[fill=black] (1,4) circle (0.05);
	\node[right] at (1,3.5) {$5_r$};
	\draw[fill=black] (1,3.5) circle (0.05);
	\node[right] at (1,3) {$6_r$};
	\draw[fill=black] (1,3) circle (0.05);
	\node[right] at (1,2.5) {$7_r$};
	\draw[red, fill=red] (1,2.5) circle (0.05);
	\node[right] at (1,2) {$8_r$};
	\draw[fill=black] (1,2) circle (0.05);
	\node[right] at (1,1.5) {$9_r$};
	\draw[red, fill=red] (1,1.5) circle (0.05);
	\end{tikzpicture}
\end{align*}

If $i_p = l_{p+1} - l_{p}$ and $j_q = k_{q+1} - k_{q}$, then $\sum^{t-1}_{p=1} i_p = n$ and $\sum^{s-1}_{q=1} j_q = m$.  Using similar arguments to those in Lemma \ref{lem:S-case-1}, expanding
\[
\kappa_\rho(\underbrace{a_2, a_1, a_2, a_1, \ldots, a_1, a_2}_{a_1 \text{ occurs }n \text{ times}}, \underbrace{b_2, b_1, b_2, b_1, \ldots, b_1, b_2}_{b_1 \text{ occurs }m \text{ times}}) z^{n+1} w^{m+1}
\]
for $\rho \in BNC_S(n,m)_{o,0}$ and summing all terms with $V_{\rho, \ell} = V_{\pi, \ell}$ and $V_{\rho, r} = V_{\pi, r}$, we obtain 
\[
zw \cdot \kappa_{t}(a_2) \kappa_s(b_2)\left(\prod^{t-1}_{p=1} (f_2 \check{\ast} f_1)(0_{i_p}, 1_{i_p}) z^{i_p} \right) \left(\prod^{s-1}_{q=1} (g_2 \check{\ast} g_1)(0_{j_q}, 1_{j_q}) w^{j_q}\right).
\]
Finally, if we sum over all possible $n,m\geq 0$ and all possible $V_{\pi, \ell}$ and $V_{\pi, r}$ (so, in the above equation, we get all possible $t, s\geq 1$ and all possible $i_p, j_q\geq 1$), we obtain that 
\begin{align*}
\Psi_e(z,w) &= zw\sum_{t,s \geq 1} \kappa_{t}(a_2) \kappa_s(b_2)\left(\prod^{t-1}_{p=1} \phi_{f_2 \check{\ast} f_1}(z) \right) \left(\prod^{s-1}_{q=1} \phi_{g_2 \check{\ast} g_1}(z)\right)  \\
& = zw\sum_{t,s \geq 1} \kappa_{t}(a_2) \kappa_s(b_2) \left(\phi_{f_2 \check{\ast} f_1}(z)\right)^{t-1} \left(\phi_{g_2 \check{\ast} g_1}(w)\right)^{s-1} = zw\cdot \frac{\phi_{f_2}\left( \phi_{f_2 \check{\ast} f_1}(z)\right)\phi_{g_2}\left( \phi_{g_2 \check{\ast} g_1}(w)\right)}{\phi_{f_2 \check{\ast} f_1}(z) \phi_{g_2 \check{\ast} g_1}(w)}. \qedhere
\end{align*}

\end{proof}

\begin{lem}
\label{lem:S-case-3}
Under the above notation and assumptions, 
\[
\Psi_{o,r}(z,w) = \frac{w \cdot \phi_{f_1 \check{\ast} f_2}(z)}{\phi_{g_2 \check{\ast} g_1}(w)} K_{a_2, b_2}\left(\phi_{f_2 \check{\ast} f_1}(z), \phi_{g_2 \check{\ast} g_1}(w)\right).
\]
\end{lem}
\begin{proof}
Fix $n,m \geq 0$.  Note $BNC_S(0,m)_{o,r} = \emptyset$   by definition.

If $\pi \in BNC_S(n,m)_{o,r}$, then, since $\pi \vee \sigma'_{n,m} = 1_{2n+1,2m+1}$, there exist $t,s\geq 1$, $1 < l_1 < l_2 < \cdots < l_t = n+1$, and $1 = k_1 < k_2 < \cdots < k_s = m+1$ such that
\[
V_{\pi, r} = \{(2l_p-1)_\ell\}^t_{p=1}  \cup  \{(2k_q-1)_r\}^s_{q=1}.
\]
Note $V_{\pi, r}$ divides the remaining right nodes into $s-1$ disjoint regions and the remaining left nodes into $t$ regions.  However, the top region is special.  If $l_0$ is the largest natural number such that $(2l_0 -1)_\ell \in V_{\pi, \ell}$, then $l_0$ further divides the top region on the left into two regions. Note each block of $\pi$ can only contain nodes in one such region.
The following is an example of such a $\pi$ for which $l_0 = 3$, with one part of the special region ($1_\ell, \ldots, 5_\ell$) shaded differently.
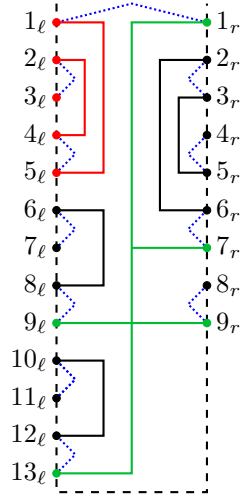
\begin{align*}
	\begin{tikzpicture}[baseline]
	\draw[thick, dashed] (-1,5.75) -- (-1,-.75) -- (1,-.75) -- (1,5.75);
	\draw[thick, blue, densely dotted] (1, 5.5) -- (0, 5.75) -- (-1, 5.5);
	\draw[thick, blue, densely dotted] (1, 5) -- (0.75, 4.75) -- (1, 4.5);
	\draw[thick, blue, densely dotted] (1, 4) -- (0.75, 3.75) -- (1, 3.5);
	\draw[thick, blue, densely dotted] (1, 3) -- (0.75, 2.75) -- (1, 2.5);
	\draw[thick, blue, densely dotted] (1, 2) -- (0.75, 1.75) -- (1, 1.5);
	\draw[thick, blue, densely dotted] (-1, 1) -- (-0.75, 0.75) -- (-1, 0.5);
	\draw[thick, blue, densely dotted] (-1, 5) -- (-0.75, 4.75) -- (-1, 4.5);
	\draw[thick, blue, densely dotted] (-1, 4) -- (-0.75, 3.75) -- (-1, 3.5);
	\draw[thick, blue, densely dotted] (-1, 3) -- (-0.75, 2.75) -- (-1, 2.5);
	\draw[thick, blue, densely dotted] (-1, 2) -- (-0.75, 1.75) -- (-1, 1.5);
	\draw[thick, blue, densely dotted] (-1, 1) -- (-0.75, 0.75) -- (-1, 0.5);
	\draw[thick, blue, densely dotted] (-1, 0) -- (-0.75, -0.25) -- (-1, -0.5);
	\draw[thick] (1, 5) -- (.375, 5) -- (.375, 3) -- (1, 3);	
	\draw[thick] (1, 4.5) -- (.625, 4.5) -- (.625, 3.5) -- (1, 3.5);
	\draw[thick, ggreen] (1, 5.5) -- (0, 5.5) -- (0, -.5) -- (-1, -.5);
	\draw[thick, ggreen] (-1, 1.5) -- (1, 1.5);
	\draw[thick, ggreen] (1, 2.5) -- (0, 2.5);
	\draw[thick] (-1,0) -- (-.375, 0) -- (-.375, 1) -- (-1,1);
	\draw[thick] (-1,2) -- (-.375, 2) -- (-.375, 3) -- (-1,3);
	\draw[thick, red] (-1, 5.5) -- (-.375, 5.5) -- (-.375, 3.5) -- (-1, 3.5);
	\draw[thick, red] (-1, 5) -- (-.625, 5) -- (-.625, 4) -- (-1, 4);
	\node[left] at (-1, 5.5) {$1_\ell$};
	\draw[red, fill=red] (-1,5.5) circle (0.05);
	\node[left] at (-1, 5) {$2_\ell$};
	\draw[red, fill=red] (-1,5) circle (0.05);
	\node[left] at (-1, 4.5) {$3_\ell$};
	\draw[red, fill=red] (-1,4.5) circle (0.05);
	\node[left] at (-1, 4) {$4_\ell$};
	\draw[red, fill=red] (-1,4) circle (0.05);
	\node[left] at (-1, 3.5) {$5_\ell$};
	\draw[red, fill=red] (-1,3.5) circle (0.05);
	\node[left] at (-1, 3) {$6_\ell$};
	\draw[fill=black] (-1,3) circle (0.05);
	\node[left] at (-1, 2.5) {$7_\ell$};
	\draw[fill=black] (-1,2.5) circle (0.05);
	\node[left] at (-1, 2) {$8_\ell$};
	\draw[fill=black] (-1,2) circle (0.05);
	\node[left] at (-1, 1.5) {$9_\ell$};
	\draw[ggreen, fill=ggreen] (-1,1.5) circle (0.05);
	\node[left] at (-1, 1) {$10_\ell$};
	\draw[fill=black] (-1,1) circle (0.05);
	\node[left] at (-1, .5) {$11_\ell$};
	\draw[fill=black] (-1,.5) circle (0.05);
	\node[left] at (-1, 0) {$12_\ell$};
	\draw[fill=black] (-1,0) circle (0.05);
	\node[left] at (-1, -.5) {$13_\ell$};
	\draw[ggreen, fill=ggreen] (-1,-.5) circle (0.05);
	\draw[ggreen, fill=ggreen] (1,5.5) circle (0.05);
	\node[right] at (1,5.5) {$1_r$};
	\draw[fill=black] (1,5) circle (0.05);
	\node[right] at (1,5) {$2_r$};
	\draw[fill=black] (1,4.5) circle (0.05);
	\node[right] at (1,4.5) {$3_r$};
	\node[right] at (1,4) {$4_r$};
	\draw[fill=black] (1,4) circle (0.05);
	\node[right] at (1,3.5) {$5_r$};
	\draw[fill=black] (1,3.5) circle (0.05);
	\node[right] at (1,3) {$6_r$};
	\draw[fill=black] (1,3) circle (0.05);
	\node[right] at (1,2.5) {$7_r$};
	\draw[ggreen, fill=ggreen] (1,2.5) circle (0.05);
	\node[right] at (1,2) {$8_r$};
	\draw[fill=black] (1,2) circle (0.05);
	\node[right] at (1,1.5) {$9_r$};
	\draw[ggreen, fill=ggreen] (1,1.5) circle (0.05);
	\end{tikzpicture}
\end{align*}

Let $i_0 = l_0$, $i_p = l_{p} - l_{p-1}$ when $p \neq 0$, and $j_q = k_{q+1} - k_{q}$.  Thus $\sum^t_{p=0} i_p = n+1$ and $\sum^{s-1}_{q=1} j_q = m$.  Using similar arguments to those in Lemma \ref{lem:S-case-1}, expanding
\[
\kappa_\rho(\underbrace{a_2, a_1, a_2, a_1, \ldots, a_1, a_2}_{a_1 \text{ occurs }n \text{ times}}, \underbrace{b_2, b_1, b_2, b_1, \ldots, b_1, b_2}_{b_1 \text{ occurs }m \text{ times}}) z^{n+1} w^{m+1}
\]
for $\rho \in BNC_S(n,m)_{o,r}$ and summing all terms with $V_{\rho, \ell} = V_{\pi, \ell}$ and $V_{\rho, r} = V_{\pi, r}$, we obtain 
\[
w \cdot \kappa_{t, s}(a_2, b_2)\left(\prod^{t}_{p=1} (f_2 \check{\ast} f_1)(0_{i_p}, 1_{i_p}) z^{i_p} \right) \left(\prod^{s-1}_{q=1} (g_2 \check{\ast} g_1)(0_{j_q}, 1_{j_q}) w^{j_q}\right) \left((f_1 \check{\ast} f_2)(0_{i_0}, 1_{i_0})z^{i_0}\right).
\]
Note for $p \geq 2$, each $(f_2 \check{\ast} f_1)(0_{i_p}, 1_{i_p}) z^{i_p}$ comes from the $p^{\mathrm{th}}$ region  from the top on the left, where as the top region on the left gives $(f_2 \check{\ast} f_1)(0_{i_1}, 1_{i_1}) z^{i_1} $ using the partitions below $(2l_0-1)_\ell$ and gives $(f_1 \check{\ast} f_2)(0_{i_0}, 1_{i_0})z^{i_0}$ using the partitions above and including $(2l_0-1)_\ell$.

Finally, if we sum over all possible $n,m\geq 0$ and all possible $V_{\pi, \ell}$ and $V_{\pi, r}$ (so, in the above equation, we get all possible $t, s\geq 1$ and all possible $i_p, j_q\geq 1$), we obtain that 
\begin{align*}
\Psi_e(z,w) &= w\sum_{t,s \geq 1} \kappa_{t, s}(a_2, b_2)\left(\prod^{t}_{p=1} \phi_{f_2 \check{\ast} f_1}(z) \right) \left(\prod^{s-1}_{q=1} \phi_{g_2 \check{\ast} g_1}(z)\right) \left(\phi_{f_1 \check{\ast} f_2}(z) \right) \\
& = w\sum_{t,s \geq 1} \kappa_{t, s}(a_2, b_2) \left(\phi_{f_2 \check{\ast} f_1}(z)\right)^{t} \left(\phi_{g_2 \check{\ast} g_1}(w)\right)^{s-1} \left(\phi_{f_1 \check{\ast} f_2}(z) \right) \\
&= \frac{w \cdot \phi_{f_1 \check{\ast} f_2}(z)}{\phi_{g_2 \check{\ast} g_1}(w)} K_{a_2, b_2}\left(\phi_{f_2 \check{\ast} f_1}(z), \phi_{g_2 \check{\ast} g_1}(w)\right). \qedhere
\end{align*}

\end{proof}

\begin{lem}
\label{lem:S-case-4}
Under the above notation and assumptions,
\[
\Psi_{o,\ell}(z,w) = \frac{z \cdot \phi_{g_1 \check{\ast} g_2}(w)}{\phi_{f_2 \check{\ast} f_1}(z)} K_{a_2, b_2}\left(\phi_{f_2 \check{\ast} f_1}(z), \phi_{g_2 \check{\ast} g_1}(w)\right) .
\]
\end{lem}
\begin{proof}
The proof of this result can be obtained by applying a mirror to Lemma \ref{lem:S-case-3}.
\end{proof}

\begin{lem}
\label{lem:S-case-5}
Under the above notation and assumptions,
\[
\Psi_{o, \ell r}(z,w) = \frac{zw}{\phi_{f_2 \check{\ast} f_1}(z) \phi_{g_2 \check{\ast} g_1}(w)} K_{a_2, b_2}\left(\phi_{f_2 \check{\ast} f_1}(z),  \phi_{g_2 \check{\ast} g_1}(w)  \right).
\]
\end{lem}
\begin{proof}
The proof of this result follows from the proof of Lemma \ref{lem:S-case-2} by replacing each occurrence of $\kappa_{t}(a_2) \kappa_s(b_2)$ with $\kappa_{t,s}(a_2, b_2)$.  Indeed there is a bijection from $BNC_S(n,m)_{o, 0}$ to $BNC_S(n,m)_{o,\ell r}$ where given $\pi \in BNC_S(n,m)_{o,0}$ we produce $\pi'\in BNC_S(n,m)_{o,\ell r}$ by joining $V_{\pi, \ell}$ and $V_{\pi, r}$ into a single block. 
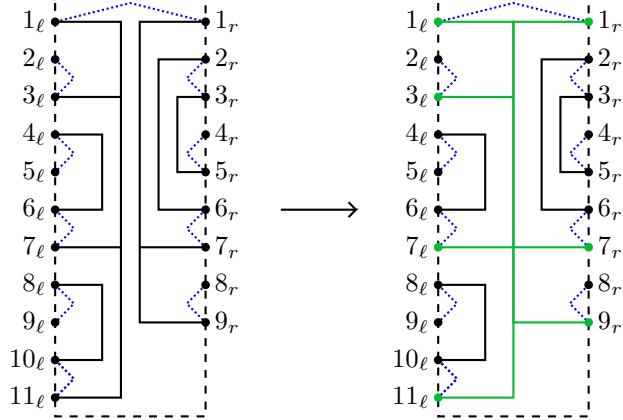
\begin{align*}
	\begin{tikzpicture}[baseline]
	\draw[thick, blue, densely dotted] (1, 5.5) -- (0, 5.75) -- (-1, 5.5);
	\draw[thick, blue, densely dotted] (1, 5) -- (0.75, 4.75) -- (1, 4.5);
	\draw[thick, blue, densely dotted] (1, 4) -- (0.75, 3.75) -- (1, 3.5);
	\draw[thick, blue, densely dotted] (1, 3) -- (0.75, 2.75) -- (1, 2.5);
	\draw[thick, blue, densely dotted] (1, 2) -- (0.75, 1.75) -- (1, 1.5);
	\draw[thick, blue, densely dotted] (-1, 1) -- (-0.75, 0.75) -- (-1, 0.5);
	\draw[thick, blue, densely dotted] (-1, 5) -- (-0.75, 4.75) -- (-1, 4.5);
	\draw[thick, blue, densely dotted] (-1, 4) -- (-0.75, 3.75) -- (-1, 3.5);
	\draw[thick, blue, densely dotted] (-1, 3) -- (-0.75, 2.75) -- (-1, 2.5);
	\draw[thick, blue, densely dotted] (-1, 2) -- (-0.75, 1.75) -- (-1, 1.5);
	\draw[thick, blue, densely dotted] (-1, 1) -- (-0.75, 0.75) -- (-1, 0.5);
	\draw[thick] (1, 5) -- (.375, 5) -- (.375, 3) -- (1, 3);	
	\draw[thick] (1, 4.5) -- (.625, 4.5) -- (.625, 3.5) -- (1, 3.5);
	\draw[thick] (1, 5.5) -- (.125, 5.5) -- (.125, 1.5) -- (1, 1.5);
	\draw[thick] (-1, 5.5) -- (-.125, 5.5) -- (-.125, 0.5) -- (-1, 0.5);
	\draw[thick] (-1, 4.5) -- (-.125, 4.5);
	\draw[thick] (-1, 2.5) -- (-.125, 2.5);
	\draw[thick] (1, 2.5) -- (.125, 2.5);
	\draw[thick] (-1, 4) -- (-.375, 4) -- (-.375, 3) -- (-1, 3);
	\draw[thick] (-1, 2) -- (-.375, 2) -- (-.375, 1) -- (-1, 1);
	\node[left] at (-1, 5.5) {$1_\ell$};
	\draw[fill=black] (-1,5.5) circle (0.05);
	\node[left] at (-1, 5) {$2_\ell$};
	\draw[fill=black] (-1,5) circle (0.05);
	\node[left] at (-1, 4.5) {$3_\ell$};
	\draw[fill=black] (-1,4.5) circle (0.05);
	\node[left] at (-1, 4) {$4_\ell$};
	\draw[fill=black] (-1,4) circle (0.05);
	\node[left] at (-1, 3.5) {$5_\ell$};
	\draw[fill=black] (-1,3.5) circle (0.05);
	\node[left] at (-1, 3) {$6_\ell$};
	\draw[fill=black] (-1,3) circle (0.05);
	\node[left] at (-1, 2.5) {$7_\ell$};
	\draw[fill=black] (-1,2.5) circle (0.05);
	\node[left] at (-1, 2) {$8_\ell$};
	\draw[fill=black] (-1,2) circle (0.05);
	\node[left] at (-1, 1.5) {$9_\ell$};
	\draw[fill=black] (-1,1.5) circle (0.05);
	\node[left] at (-1, 1) {$10_\ell$};
	\draw[fill=black] (-1,1) circle (0.05);
	\node[left] at (-1, .5) {$11_\ell$};
	\draw[fill=black] (-1,.5) circle (0.05);
	\draw[fill=black] (1,5.5) circle (0.05);
	\node[right] at (1,5.5) {$1_r$};
	\draw[fill=black] (1,5) circle (0.05);
	\node[right] at (1,5) {$2_r$};
	\draw[fill=black] (1,4.5) circle (0.05);
	\node[right] at (1,4.5) {$3_r$};
	\node[right] at (1,4) {$4_r$};
	\draw[fill=black] (1,4) circle (0.05);
	\node[right] at (1,3.5) {$5_r$};
	\draw[fill=black] (1,3.5) circle (0.05);
	\node[right] at (1,3) {$6_r$};
	\draw[fill=black] (1,3) circle (0.05);
	\node[right] at (1,2.5) {$7_r$};
	\draw[fill=black] (1,2.5) circle (0.05);
	\node[right] at (1,2) {$8_r$};
	\draw[fill=black] (1,2) circle (0.05);
	\node[right] at (1,1.5) {$9_r$};
	\draw[fill=black] (1,1.5) circle (0.05);
	\draw[thick, dashed] (-1,5.75) -- (-1,.25) -- (1,.25) -- (1,5.75);
	\draw[thick] (2, 3) -- (3,3) -- (2.90, 2.90);
	\draw[thick] (3,3) -- (2.90, 3.1);
	\end{tikzpicture}
	\quad	
	\begin{tikzpicture}[baseline]
	\draw[thick, dashed] (-1,5.75) -- (-1,.25) -- (1,.25) -- (1,5.75);
	\draw[thick, blue, densely dotted] (1, 5.5) -- (0, 5.75) -- (-1, 5.5);
	\draw[thick, blue, densely dotted] (1, 5) -- (0.75, 4.75) -- (1, 4.5);
	\draw[thick, blue, densely dotted] (1, 4) -- (0.75, 3.75) -- (1, 3.5);
	\draw[thick, blue, densely dotted] (1, 3) -- (0.75, 2.75) -- (1, 2.5);
	\draw[thick, blue, densely dotted] (1, 2) -- (0.75, 1.75) -- (1, 1.5);
	\draw[thick, blue, densely dotted] (-1, 1) -- (-0.75, 0.75) -- (-1, 0.5);
	\draw[thick, blue, densely dotted] (-1, 5) -- (-0.75, 4.75) -- (-1, 4.5);
	\draw[thick, blue, densely dotted] (-1, 4) -- (-0.75, 3.75) -- (-1, 3.5);
	\draw[thick, blue, densely dotted] (-1, 3) -- (-0.75, 2.75) -- (-1, 2.5);
	\draw[thick, blue, densely dotted] (-1, 2) -- (-0.75, 1.75) -- (-1, 1.5);
	\draw[thick, blue, densely dotted] (-1, 1) -- (-0.75, 0.75) -- (-1, 0.5);
	\draw[thick] (1, 5) -- (.375, 5) -- (.375, 3) -- (1, 3);	
	\draw[thick] (1, 4.5) -- (.625, 4.5) -- (.625, 3.5) -- (1, 3.5);
	\draw[thick, ggreen] (1, 5.5) -- (0, 5.5) -- (0, 1.5) -- (1, 1.5);
	\draw[thick, ggreen] (-1, 5.5) -- (0, 5.5) -- (0, 0.5) -- (-1, 0.5);
	\draw[thick, ggreen] (-1, 4.5) -- (0, 4.5);
	\draw[thick, ggreen] (-1, 2.5) -- (0, 2.5);
	\draw[thick, ggreen] (1, 2.5) -- (0, 2.5);
	\draw[thick] (-1, 4) -- (-.375, 4) -- (-.375, 3) -- (-1, 3);
	\draw[thick] (-1, 2) -- (-.375, 2) -- (-.375, 1) -- (-1, 1);
	\node[left] at (-1, 5.5) {$1_\ell$};
	\draw[ggreen, fill=ggreen] (-1,5.5) circle (0.05);
	\node[left] at (-1, 5) {$2_\ell$};
	\draw[fill=black] (-1,5) circle (0.05);
	\node[left] at (-1, 4.5) {$3_\ell$};
	\draw[ggreen, fill=ggreen] (-1,4.5) circle (0.05);
	\node[left] at (-1, 4) {$4_\ell$};
	\draw[fill=black] (-1,4) circle (0.05);
	\node[left] at (-1, 3.5) {$5_\ell$};
	\draw[fill=black] (-1,3.5) circle (0.05);
	\node[left] at (-1, 3) {$6_\ell$};
	\draw[fill=black] (-1,3) circle (0.05);
	\node[left] at (-1, 2.5) {$7_\ell$};
	\draw[ggreen, fill=ggreen] (-1,2.5) circle (0.05);
	\node[left] at (-1, 2) {$8_\ell$};
	\draw[fill=black] (-1,2) circle (0.05);
	\node[left] at (-1, 1.5) {$9_\ell$};
	\draw[fill=black] (-1,1.5) circle (0.05);
	\node[left] at (-1, 1) {$10_\ell$};
	\draw[fill=black] (-1,1) circle (0.05);
	\node[left] at (-1, .5) {$11_\ell$};
	\draw[ggreen, fill=ggreen] (-1,.5) circle (0.05);
	\draw[ggreen, fill=ggreen] (1,5.5) circle (0.05);
	\node[right] at (1,5.5) {$1_r$};
	\draw[fill=black] (1,5) circle (0.05);
	\node[right] at (1,5) {$2_r$};
	\draw[fill=black] (1,4.5) circle (0.05);
	\node[right] at (1,4.5) {$3_r$};
	\node[right] at (1,4) {$4_r$};
	\draw[fill=black] (1,4) circle (0.05);
	\node[right] at (1,3.5) {$5_r$};
	\draw[fill=black] (1,3.5) circle (0.05);
	\node[right] at (1,3) {$6_r$};
	\draw[fill=black] (1,3) circle (0.05);
	\node[right] at (1,2.5) {$7_r$};
	\draw[ggreen, fill=ggreen] (1,2.5) circle (0.05);
	\node[right] at (1,2) {$8_r$};
	\draw[fill=black] (1,2) circle (0.05);
	\node[right] at (1,1.5) {$9_r$};
	\draw[ggreen, fill=ggreen] (1,1.5) circle (0.05);
	\end{tikzpicture}
\end{align*}
\end{proof}

\begin{lem}
\label{lem:S-case-6}
Under the above notation and assumptions, 
\[
\Psi_{o}(z,w) = \frac{1}{\phi_{f_1 \check{\ast} f_2}(z) \phi_{g_1 \check{\ast} g_2}(w)} \Psi_{o'}(z,w) K_{a_1, b_1}\left( \phi_{f_1 \check{\ast} f_2}(z), \phi_{g_1 \check{\ast} g_2}(w)   \right)
\]
where
\[
\Psi_{o'}(z,w) = \Psi_{o,0}(z,w) + \Psi_{o,r}(z,w) + \Psi_{o,\ell}(z,w) + \Psi_{o,\ell r}(z,w).
\]
\end{lem}
\begin{proof}

Fix $n,m \geq 1$.  If $\pi \in BNC_S(n,m)_o$, let $V_\pi$ denote the first block of $\pi$, as measured from the top of $\pi$'s bi-non-crossing diagram, that has both left and right nodes.  Since $\pi \in BNC_S(n,m)_o$, there exist $t,s\geq 1$, $1= l_1 < l_2 < \cdots < l_t \leq n$, and $1 = k_1 < k_2 < \cdots < k_s \leq m$ such that
\[
V_\pi = \{(2l_p-1)_\ell\}^t_{p=1} \cup \{(2k_q-1)_r\}^s_{q=1}.
\]
Note $V_\pi$ divides the remaining left nodes and right nodes into $t-1$ disjoint regions on the left, $s-1$ disjoint regions on the right, and one region on the bottom. Moreover, each block of $\pi$ can only contain nodes in one such region.   Below is an example of such a $\pi$.
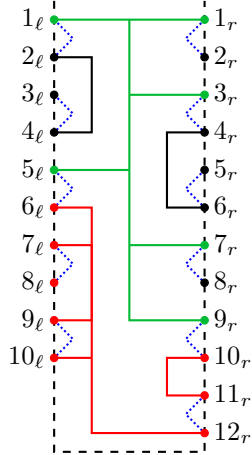
\begin{align*}
	\begin{tikzpicture}[baseline]
	\draw[thick, dashed] (-1,5.75) -- (-1,-.25) -- (1,-.25) -- (1,5.75);
	\draw[thick, blue, densely dotted] (1, 5.5) -- (0.75, 5.25) -- (1, 5);
	\draw[thick, blue, densely dotted] (1, 4.5) -- (0.75, 4.25) -- (1, 4);
	\draw[thick, blue, densely dotted] (1, 3.5) -- (0.75, 3.25) -- (1, 3);
	\draw[thick, blue, densely dotted] (1, 2.5) -- (0.75, 2.25) -- (1, 2);
	\draw[thick, blue, densely dotted] (1, 1.5) -- (0.75, 1.25) -- (1, 1);
	\draw[thick, blue, densely dotted] (1, 0.5) -- (0.75, 0.25) -- (1, 0);
	\draw[thick, blue, densely dotted] (-1, 5.5) -- (-0.75, 5.25) -- (-1, 5);
	\draw[thick, blue, densely dotted] (-1, 4.5) -- (-0.75, 4.25) -- (-1, 4);
	\draw[thick, blue, densely dotted] (-1, 3.5) -- (-0.75, 3.25) -- (-1, 3);
	\draw[thick, blue, densely dotted] (-1, 2.5) -- (-0.75, 2.25) -- (-1, 2);
	\draw[thick, blue, densely dotted] (-1, 1.5) -- (-0.75, 1.25) -- (-1, 1);
	\draw[thick, ggreen] (-1,5.5) -- (1,5.5);
	\draw[thick, ggreen] (0, 5.5) -- (0, 1.5) -- (1, 1.5);
	\draw[thick, ggreen] (1,2.5) -- (0, 2.5);
	\draw[thick, ggreen] (1, 4.5) -- (0, 4.5);
	\draw[thick, ggreen] (-1, 3.5) -- (0, 3.5);
	\draw[thick] (-1, 5) -- (-0.5, 5) -- (-0.5, 4) -- (-1, 4);
	\draw[thick] (1, 4) -- (.5, 4) -- (.5, 3) -- (1, 3);
	\draw[thick, red] (1, 1) -- (.5, 1) -- (.5, .5) -- (1, .5);
	\draw[thick, red] (-1, 3) -- (-.5, 3) -- (-.5, 0) -- (1, 0);
	\draw[thick, red] (-1, 1.5) -- (-.5, 1.5) -- (-.5, 2.5) -- (-1, 2.5);
	\draw[thick, red] (-1, 1) -- (-.5, 1);
	\node[left] at (-1, 5.5) {$1_\ell$};
	\draw[ggreen, fill=ggreen] (-1,5.5) circle (0.05);
	\node[left] at (-1, 5) {$2_\ell$};
	\draw[fill=black] (-1,5) circle (0.05);
	\node[left] at (-1, 4.5) {$3_\ell$};
	\draw[fill=black] (-1,4.5) circle (0.05);
	\node[left] at (-1, 4) {$4_\ell$};
	\draw[fill=black] (-1,4) circle (0.05);
	\node[left] at (-1, 3.5) {$5_\ell$};
	\draw[ggreen, fill=ggreen] (-1,3.5) circle (0.05);
	\node[left] at (-1, 3) {$6_\ell$};
	\draw[red, fill=red] (-1,3) circle (0.05);
	\node[left] at (-1, 2.5) {$7_\ell$};
	\draw[red, fill=red] (-1,2.5) circle (0.05);
	\node[left] at (-1, 2) {$8_\ell$};
	\draw[red, fill=red] (-1,2) circle (0.05);
	\node[left] at (-1, 1.5) {$9_\ell$};
	\draw[red, fill=red] (-1,1.5) circle (0.05);
	\node[left] at (-1, 1) {$10_\ell$};
	\draw[red, fill=red] (-1,1) circle (0.05);
	\draw[ggreen, fill=ggreen] (1,5.5) circle (0.05);
	\node[right] at (1,5.5) {$1_r$};
	\draw[fill=black] (1,5) circle (0.05);
	\node[right] at (1,5) {$2_r$};
	\draw[ggreen, fill=ggreen] (1,4.5) circle (0.05);
	\node[right] at (1,4.5) {$3_r$};
	\draw[red, fill=red] (1,0) circle (0.05);
	\node[right] at (1,4) {$4_r$};
	\draw[fill=black] (1,4) circle (0.05);
	\node[right] at (1,3.5) {$5_r$};
	\draw[fill=black] (1,3.5) circle (0.05);
	\node[right] at (1,3) {$6_r$};
	\draw[fill=black] (1,3) circle (0.05);
	\node[right] at (1,2.5) {$7_r$};
	\draw[ggreen, fill=ggreen] (1,2.5) circle (0.05);
	\node[right] at (1,2) {$8_r$};
	\draw[fill=black] (1,2) circle (0.05);
	\node[right] at (1,1.5) {$9_r$};
	\draw[ggreen, fill=ggreen] (1,1.5) circle (0.05);
	\node[right] at (1,1) {$10_r$};
	\draw[red, fill=red] (1,1) circle (0.05);
	\node[right] at (1,0.5) {$11_r$};
	\draw[red, fill=red] (1,0.5) circle (0.05);
	\node[right] at (1,0) {$12_r$};
	\end{tikzpicture}
\end{align*}

Let $E = \{(2k)_\ell\}^{n}_{k=1} \cup \{(2k)_r\}^{m}_{k=1}$ and let $O = \{(2k-1)_\ell\}^{n}_{k=1} \cup \{(2k-1)_r\}^{m}_{k=1}$.  For each $1 \leq p \leq t$, let $i_p = l_{p+1} - l_{p}$, where $l_{t+1} = n+1$, and, for $p \neq t$, let $\pi_{\ell, p}$ denote the non-crossing partition obtained by restricting $\pi$ to $\{(2l_{p})_\ell, (2l_{p}+1)_\ell, \ldots, (2l_{p+1}-2)_\ell\}$.  Note that $\sum^t_{p=1} i_p = n$.  Furthermore, as explained in Lemma \ref{lem:T-case-1}, if $\pi'_{\ell, p}$ is obtained from $\pi_{\ell, p}$ by adding the singleton block $\{(2l_p-1)_\ell\}$, then $\pi'_{\ell, p}|_{O}$ is naturally an element of $NC'(i_p)$ and $\pi'_{\ell, p}|_E$ is naturally an element of $NC(i_p)$, which must be $K(\pi'_{\ell, p}|_O)$ in order for $\pi \vee \sigma_{n,m} = 1_{2n,2m}$. 

Similarly, for each $1 \leq q \leq s$, let $j_q = k_{q+1} - k_{q}$, where $k_{s+1} = m+1$, and, for $q \neq s$, let $\pi_{r, q}$ denote the non-crossing partition obtained by restricting $\pi$ to $\{(2k_{q})_r, (2k_{q}+1)_r, \ldots, (2k_{q+1}-2)_r\}$.  Note that $\sum^s_{q=1} j_q = m$.  Furthermore, as explained in Lemma \ref{lem:T-case-1}, if $\pi'_{r, q}$ is obtained from $\pi_{r, q}$ by adding the singleton block $\{(2k_q-1)_r\}$, then $\pi'_{r, q}|_{O}$ is naturally an element of $NC'(j_q)$ and $\pi'_{r, q}|_E$ is naturally an element of $NC(j_q)$, which must be $K(\pi'_{r, q}|_O)$ in order for $\pi \vee \sigma_{n,m} = 1_{2n,2m}$. 

Finally, if $\pi'$ is the bi-non-crossing partition obtained by restricting $\pi$ to 
\[
\{(2l_{t})_\ell, (2l_{t}+1)_\ell, \ldots, (2n)_\ell, (2k_{s})_r, (2k_{s}+1)_r, \ldots, (2m)_r\}
\]
(which is shaded differently in the above diagram), then $\pi' \in BNC_S(i_t-1,j_s-1)'_o$.

Expanding
\[
\kappa_\rho(\underbrace{a_1, a_2, \ldots, a_1, a_2}_{a_1 \text{ occurs }n \text{ times}}, \underbrace{b_1, b_2, \ldots, b_1, b_2}_{b_1 \text{ occurs }m \text{ times}}) z^n w^m
\]
for $\rho \in BNC_S(n,m)_o$ and  summing such terms with $V_\rho = V_\pi$, we obtain 
\begin{align*}
\kappa_{t, s}(a_1, b_1)&\left(\prod^{t-1}_{p=1} (f_1 \check{\ast} f_2)(0_{i_p}, 1_{i_p}) z^{i_p} \right) \left(\prod^{s-1}_{q=1} (g_1 \check{\ast} g_2)(0_{j_q}, 1_{j_q}) w^{j_q}\right)\\
&\cdot \left(\sum_{\tau \in  BNC_S(i_t-1,j_s-1)'_o} \kappa_\tau(\underbrace{a_2, a_1, a_2, a_1, \ldots, a_1, a_2}_{a_1 \text{ occurs }i_t-1 \text{ times}}, \underbrace{b_2, b_1, b_2, b_1, \ldots, b_1, b_2}_{b_1 \text{ occurs }j_s-1 \text{ times}})   z^{i_t} w^{j_s} \right).
\end{align*}
Note for $p \neq t$ each $(f_1 \check{\ast} f_2)(0_{i_p}, 1_{i_p}) z^{i_p}$ comes from the $p^{\mathrm{th}}$ region from the top on the left, for $q \neq s$ each $(g_1 \check{\ast} g_2)(0_{j_q}, 1_{j_q}) w^{j_q}$ comes from the $q^{\mathrm{th}}$ region from the top on the right, and all $\tau \in BNC_S(i_t-1,j_s-1)'_o$ are possible on the bottom, with the coefficient the sum being the correct one

Finally, if we sum over all possible $n,m\geq 1$ and all possible $V_\pi$ (so, in the above equation, we get all possible $t, s\geq 1$ and all possible $i_p, j_q\geq 1$), we obtain that 
\begin{align*}
\Psi_e(z,w) &= \sum_{t,s \geq 1} \kappa_{t, s}(a_1, b_1)\left(\prod^{t-1}_{p=1} \phi_{f_1 \check{\ast} f_2}(z) \right) \left(\prod^{s-1}_{q=1} \phi_{g_1 \check{\ast} g_2}(z)\right) \Psi_{o'}(z,w) \\
& = \sum_{t,s \geq 1} \kappa_{t, s}(a_1, b_1) \left(\phi_{f_1 \check{\ast} f_2}(z)\right)^{t-1} \left(\phi_{g_1 \check{\ast} g_2}(w)\right)^{s-1}\Psi_{o'}(z,w)\\
&= \frac{1}{\phi_{f_1 \check{\ast} f_2}(z) \phi_{g_1 \check{\ast} g_2}(w)} \Psi_{o'}(z,w) K_{a_1, b_1}\left( \phi_{f_1 \check{\ast} f_2}(z), \phi_{g_1 \check{\ast} g_2}(w)   \right). \qedhere
\end{align*}

\end{proof}

\begin{proof}[Proof of Theorem \ref{thm:S-property}]
Using equations (\ref{eq:inversion-with-convolution}, \ref{eq:convolution-with-inverse-series}), we see (via Lemmata \ref{lem:S-case-1}, \ref{lem:S-case-2}, \ref{lem:S-case-3}, \ref{lem:S-case-4}, \ref{lem:S-case-5})  that
\begin{align*}
\Psi_e\left(\phi^{\inv}_{f_1 \ast f_2}(z), \phi^{\inv}_{g_1 \ast g_2}(w) \right)&=  K_{a_2, b_2}\left(\phi^{\inv}_{f_2}(z), \phi^{\inv}_{g_2}(w) \right), \\
\Psi_{o,0}\left(\phi^{\inv}_{f_1 \ast f_2}(z), \phi^{\inv}_{g_1 \ast g_2}(w) \right)&= \phi^{\inv}_{f_1 \ast f_2}(z)\phi^{\inv}_{g_1 \ast g_2}(w)\cdot \frac{zw}{\phi^{\inv}_{f_2}(z) \phi^{\inv}_{g_2}(w)} = \phi^{\inv}_{f_1}(z) \phi^{\inv}_{g_1}(w),\\
\Psi_{o,r}\left(\phi^{\inv}_{f_1 \ast f_2}(z), \phi^{\inv}_{g_1 \ast g_2}(w) \right)&= \frac{\phi^{\inv}_{f_1}(z) \phi^{\inv}_{g_1}(w)}{w} K_{a_2, b_2}\left(\phi^{\inv}_{f_2}(z), \phi^{\inv}_{g_2}(w) \right), \\
\Psi_{o,\ell}\left(\phi^{\inv}_{f_1 \ast f_2}(z), \phi^{\inv}_{g_1 \ast g_2}(w) \right)&= \frac{\phi^{\inv}_{f_1}(z) \phi^{\inv}_{g_1}(w)}{z} K_{a_2, b_2}\left(\phi^{\inv}_{f_2}(z), \phi^{\inv}_{g_2}(w) \right), \text{ and}\\
\Psi_{o,\ell r}\left(\phi^{\inv}_{f_1 \ast f_2}(z), \phi^{\inv}_{g_1 \ast g_2}(w) \right) &=\frac{\phi^{\inv}_{f_1}(z) \phi^{\inv}_{g_1}(w)}{zw} K_{a_2, b_2}\left(\phi^{\inv}_{f_2}(z), \phi^{\inv}_{g_2}(w) \right).
\end{align*}
Since
\[
\Phi_0\left(\phi^{\inv}_{f_1 \ast f_2}(z), \phi^{\inv}_{g_1 \ast g_2}(w) \right) = \frac{1}{\phi^{\inv}_{f_1}(z) \phi^{\inv}_{g_1}(w)} \Psi_{o'}\left(\phi^{\inv}_{f_1 \ast f_2}(z), \phi^{\inv}_{g_1 \ast g_2}(w) \right)  K_{a_1, b_1}\left(\phi^{\inv}_{f_1}(z), \phi^{\inv}_{g_1}(w) \right)
\]
by using equations (\ref{eq:inversion-with-convolution}) and Lemma \ref{lem:S-case-6}, and since
\[
\frac{1}{z} + \frac{1}{w} + \frac{1}{zw} = \frac{1+z+w}{zw} \qqand K_{a_1a_2, b_1b_2}(z,w) = \Psi_e(z,w) + \Psi_0(z,w),
\]
we have verified equation (\ref{eq:S-eq-to-verify}) holds and thus the proof is complete.
\end{proof}

\end{document}